\documentclass[11pt]{article}

\usepackage{amssymb}
\usepackage{graphics}

\newcommand{\hide}[1]{}

\newcommand{\qed}{$\;\;\;\Box$}
\newenvironment{proof}{\par\smallbreak{\sl\bf Proof.~}}
{\unskip\nobreak\hfill \qed \par\medbreak}

\newcounter{claim}
\renewcommand{\theclaim}{\arabic{claim}}
\newenvironment{claim}{\refstepcounter{claim}%
	\par\medskip\par\noindent{\bf Claim~\theclaim.}\rm}%
{\par\medskip\par}

\newenvironment{subproof}{\par\noindent{\bf Proof of Claim.}}%
{\qed\par\smallbreak}
\newcommand{\E}{{\cal E}}

\newcommand{\N}{{\mathbb N}}
\newcommand{\R}{{\mathbb R}}

\newcommand{\Z}{{\mathbb Z}}

\newcommand{\F}{{\cal F}}
\renewcommand{\H}{{\cal H}}

\newcommand{\CC}{{\cal C}}

\newcommand{\LL}{{\cal L}}

\newcommand{\beq}{\begin{equation}}
	\newcommand{\ee}{\end{equation}}

\renewcommand{\d}{\partial}

\newtheorem{thm}{Theorem}[section]
\newtheorem{lem}[thm]{Lemma}

\newtheorem{defn}[thm]{Definition}
\newtheorem{cor}[thm]{Corollary}
\newtheorem{rem}[thm]{Remark}

\newcommand{\al}{\alpha}

\newcommand{\eps}{\varepsilon}
\newcommand{\vphi}{\varphi}

\newcommand{\om}{\omega}
\newcommand{\io}{\iota}

\newcommand{\reff}[1]{(\ref{#1})}      

\newcommand{\sgn}{\mathop{\rm sgn}}

\newcommand{\dd}{\!\;\mathrm{d}}

\date{}

\title{
	Higher Regularity of Time-Periodic Solutions to Nonautonomous Hyperbolic Problems: Away from  Resonances
} 
\newcounter{thesame}
\author{
	Irina Kmit
	\thanks{Institute of Mathematics, Humboldt University of Berlin, Unter den Linden 6, D-10099 Berlin. On leave from the
		Institute for Applied Problems of Mechanics and Mathematics,
		Ukrainian National Academy of Sciences. 
		{\small   E-mail:
			{\tt irina.kmit@hu-berlin.de}
	}}
	\ \ \ Lutz Recke \thanks{Institute of Mathematics, Humboldt University of Berlin, Unter den Linden 6, D-10099 Berlin.
		{\small   E-mail:
			{\tt lutz.recke@hu-berlin.de}}
}}

\begin{document}

\maketitle

\noindent
\begin{abstract}
We study higher regularity and its relation to nonresonant behavior for time-periodic solutions of boundary value problems for one-dimensional linear and nonlinear nonautonomous first-order integro-differential strictly hyperbolic systems. The boundary conditions
include integral operators and various types of boundary reflections. We prove
that continuous and classical solutions have $C^k$-regularity, provided the
coefficients are sufficiently smooth and a suitable number of nonresonance
conditions is satisfied. In the linear case, these conditions involve the
principal coefficients, the diagonal lower-order coefficients, and the boundary
reflection coefficients. In the nonlinear case, they also depend on the
nonlinearities and on the solution itself.
For nonautonomous hyperbolic systems, higher regularity generally requires
additional nonresonance conditions, whose number depends on the desired order
of differentiability. These conditions are not only sufficient but, in general,
also necessary, revealing a distinctive feature of nonautonomous hyperbolic
PDEs. By contrast, in the autonomous case, a single nonresonance condition (if one is needed at all) suffices to obtain arbitrarily high regularity. We also identify a class of nonautonomous hyperbolic problems for which no nonresonance conditions are required. In this case, the higher regularity of solutions is determined solely by the regularity of the data. The main
technical tool underlying the  proofs is an abstract regularity principle formulated in the setting of vector spaces.
	
\end{abstract}  

{\it Keywords:}   nonautonomous linear and nonlinear integro-differential hyperbolic PDEs, 
(nonlinear) boundary conditions of reflective and integral type,   regularity of   solutions, 
nonresonance conditions

\section{Introduction}
\label{Introduction}

\renewcommand{\theequation}{{\thesection}.\arabic{equation}}
\setcounter{equation}{0}

\subsection{Framework and main results}\label{results}

\paragraph{Our setting.}
We investigate higher regularity of time-periodic continuous and classical solutions of first-order hyperbolic partial differential equations, establish sufficient nonresonance conditions for such regularity, and provide  examples showing that these conditions are, in general, essential.
Specifically, we consider
 boundary value problems for one-dimensional linear and nonlinear strictly hyperbolic systems originally written in Riemann invariants.

In the {\it linear} setting, we study an integro-differential hyperbolic system of the form
\beq\label{lin}
\partial_tu_j+a_j(x,t)\partial_xu_j+\sum_{k=1}^{n}b_{jk}(x,t)u_k+\left[Mu\right]_j(x,t)=f_j(x,t),\quad x\in(0,1),\ j\le n,
\ee
subjected to the  boundary conditions
\beq\label{bc_lin}
\begin{array}{ll}
	u_{j}(0,t)= [Ru]_j(t)+[Pu]_j(t)+h_j(t), \quad 1\le j\le m,\\ [1mm]
	u_{j}(1,t)= [Ru]_j(t)+[Pu]_j(t)+h_j(t), \quad m< j\le n,
\end{array}
\ee
and the time-periodic conditions
\beq\label{per}
u_j(x,t)=u_j(x,t+2\pi),\quad x\in(0,1),\ j\le n,
\ee
where  $n\ge 2$ and 
$0\le m\le n$ are fixed integers.
 Moreover, 
the  operator  $M\in\LL(C_{per}(\Pi;\R^n))$  is defined by
\beq\label{J}
\left[Mu\right]_j(x,t)=\sum_{k=1}^{n}\int_{0}^{x}m_{jk}(\xi,x,t)\,u_k(\xi,t)\dd\xi,\quad  j\le n,
\ee
and the operators $P, R\in\LL\left(C_{per}(\Pi;\R^n);C_{per}(\R;\R^n)\right)$ are given by
\beq\label{integral}
\left[Pu\right]_j(t)=\sum_{k=1}^n\int_0^1p_{jk}(x,t)\,u_k(x,t)\dd x,\quad  j\le n,
\ee
and
\begin{eqnarray}
	\displaystyle [Ru]_j(t)=\sum\limits_{k=m+1}^nr_{jk}(t)\, u_k(0,t)+\sum\limits_{k=1}^mr_{jk}(t)\, u_k(1,t),\quad  j\le n,
	\label{eq:R} 
\end{eqnarray}
where 
$$
\Pi=\left\{(x,t)\in\R^2\mid 0\le x\le 1\right\}
$$ 
and by $C_{per}(\Pi;\R^n)$ we denote  the vector space of all continuous maps $u : \Pi\to\R^n$
which are $2\pi$-periodic in $t$.

In the nonlinear setting, we consider {\it semilinear} first-order integro-differential hyperbolic systems of the type
\beq\label{semilin}
\partial_tu_j+a_j(x,t)\partial_xu_j=f_j\left(x,t,u,Mu\right),\quad x\in(0,1),\ j\le n,
\ee
or  {\it quasilinear}  systems
of the type
\beq\label{quasilin}
\partial_tu_j+a_j(x,t,u)\partial_xu_j=f_j\left(x,t,u,Mu\right),\quad x\in(0,1),\ j\le n,
\ee
The  systems \reff{semilin} and \reff{quasilin} will be endowed  with the nonlinear boundary conditions 
\begin{eqnarray}\label{bc}
	\begin{array}{ll}
		u_{j}(0,t)= h_j\Bigl(t, u_{m+1}(0,t),\dots,u_n(0,t), u_1(1,t),\dots,u_m(1,t),[Pu](t)\Bigr), \quad 1\le j\le m,\\ [3mm]
		u_{j}(1,t)=   h_j\Bigl(t, u_{m+1}(0,t),\dots,u_n(0,t), u_1(1,t),\dots,u_m(1,t),[Pu](t)\Bigr), \quad m< j\le n,
	\end{array}
\end{eqnarray}
 and the time-periodic conditions \reff{per}.
 Here
$h_j$ are nonlinear functions of  $n+2$ arguments.

All coefficients appearing in the above problems
are
assumed to be continuous and $2\pi$-periodic with respect to $t$.
Our goal is to establish conditions under which continuous or classical solutions  to the problems 
(\ref{lin})--\reff{per} or
\reff{semilin}, \reff{bc}, \reff{per},
or  \reff{quasilin}, \reff{bc},  \reff{per} reach $C^k$-regularity for $k\ge 1$.
This problem is related to the general question of whether bounded solutions of nonlinear evolution equations inherit the differentiability properties of the defining nonlinear operator (see, e.g., \cite{Hale_Scheurle}). In our hyperbolic setting, the corresponding question is to what extent time-periodic solutions inherit the regularity of the coefficients, nonlinearities, and boundary data.

\paragraph{Notation.} For vectors $a=(a_1,\dots,a_n)$ we use the norm $\|a\|=\max_{j\le k}|a_j|$.
Let $\Omega$ denote either the strip $\Pi$ or the real line $\R$. For $k\in\N_0$, set
$$
\begin{array}{ll}
	C_{per}^k(\Omega;\R^n)=C^k(\Omega;\R^n)\cap C_{per}(\Omega;\R^n),
\end{array}
$$
where $C_{per}^0(\Omega;\R^n)=C_{per}(\Omega;\R^n)$.
The spaces $C_{per}^k(\Omega;\R^n)$ are equipped with the norm
$$
\|u\|=\max_{|\al|\le k} \,\sup_{\Omega}\|\d^\al u\|.
$$

By $C_t^k(\Omega;\R^n)$ (and similarly $C_x^k(\Omega;\R^n)$) we denote the space of continuous
vector-valued functions $u\in C(\Omega;\R^n)$ whose classical derivatives $\d_tu,\d_t^2u,\dots,\d_t^ku$ belong to $C(\Omega;\R^n)$.

Throughout  the paper we  assume, roughly speaking, that, for every $j\le n$,
 the function $a_j$ is time-periodic, continuously differentiable in all arguments, and satisfies $|a_j|\ge c>0$.
Hence, we can introduce two sets of indices, $I_0$ and $I_1$,  by
$$
I_0=\left\{j\le n\mid a_j>0\right\}\quad\mbox{and}\quad I_1=\left\{j\le n\mid a_j<0\right\},
$$
and define $x_j$ by
\begin{equation}\label{*k}
	x_j=\left\{
	\begin{array}{rl}
		0\quad &\mbox{ if }\ j\in I_0\\
		1\quad &\mbox{ if }\ j\in I_1.
	\end{array}
	\right.
\end{equation}
Moreover, for given $j\le n$ and $(x,t)\in\Pi$, the $j$-th characteristic curve of the system \reff{lin} or  \reff{semilin}
passing through the point $(x,t)$ will be denoted by $\om_j(\xi)=\om_j(\xi,x,t)$ and
determined as the solution
to the  Cauchy problem
\begin{equation}\label{char}
	\partial_\xi\omega_j(\xi, x,t)=\frac{1}{a_j(\xi,\omega_j(\xi,x,t))},\;\quad
	\omega_j(x,x,t)=t.
\end{equation}
Thanks to the assumptions on  $a_j$, the characteristic curve $\tau=\omega_j(\xi,x,t)$ 
reaches the
boundary of $\Pi$ in two points with distinct ordinates. 
According to the notation \reff{*k},  $x_j$ denotes the abscissa of the point with the smaller ordinate. Notice that  the value of  $x_j$
does not depend on $x$ and $t$, nor on $u$ in the quasilinear setting.

 Finally, we introduce the following notation needed, in particular, to formulate our results for nonlinear problems:
 \beq\label{notation_u}
 \begin{array}{cc}
A_j(x,t)=a_j(x,t,u(x,t)),\quad A_{jt}(x,t)=\d_2a_j(x,t,u(x,t)),\quad A_{jk}(x,t)=\d_{u_k}a_j(x,t,u(x,t)),\\ [2mm]
 F_{j}(x,t)=f_j\Bigl(x,t,u(x,t),\left[Mu\right](x,t)\Bigr),\quad F_{jt}(x,t)=\d_2f_j\Bigl(x,t,u(x,t),\left[Mu\right](x,t)\Bigr),\\ [2mm]
F_{jk}(x,t)=\d_{u_k}f_j\Bigl(x,t,u(x,t),\left[Mu\right](x,t)\Bigr), 
 \end{array}
\ee
  where   $\partial_k$ here and in what follows denotes  the partial derivative with respect to the $k$-th argument.

\paragraph{Main results.}
\begin{defn}\label{cont}
	A function
	$
	u\in C_{\mathrm{per}}(\Omega;\mathbb R^n)
	$
	is called a {\rm continuous solution} to the problem \reff{lin}, \reff{bc_lin}, \reff{per} if it satisfies the integral equation obtained by integrating the system along characteristic curves, namely the equation~\reff{oper} in Subsection~\ref{1}.
\end{defn}

\begin{thm}\label{thm:lin}
	Let $l\ge 1$ be arbitrary fixed integer.
	Assume that,  for all $1\le j,k\le n$, the coefficients
	$a_j$,   $b_{jk}$,  $r_{jk}$, and $h_j$
	belong to  $C^{l}_{per}$;  $f_j$ belong to $C^{l-1}_{per}\cap C_t^l$;	$m_{jk}$   belong to  $C^{l-1}_{per}\cap C_x^l$;
	and  $p_{jk}$ belong to $C_{per}\cap C_t^l$ with $\d_xp_{jk}\in C_t^{l-1}$, each on its respective domain.
	Moreover, assume that the system \reff{lin}   is nondegenerate and strictly hyperbolic, namely
	\beq\label{aj}
	a_j(x,t)\ne 0\ \mbox{ and }\ a_j(x,t)\ne a_k(x,t)\quad\mbox{ for all } (x,t)\in\Pi \mbox{ and } j\ne k.
	\ee 
	If
	\beq\label{nonrez}
		\exp \left\{\int_{1 - x_j}^{x_j}
		\biggl(\frac{b_{jj}}{a_{j}}-i\frac{\d_ta_j}{a_{j}^2}\biggr)(\eta,\om_j(\eta,1 - x_j,t))\dd\eta\right\}\sum_{k=1}^n\left|r_{jk}\left(\omega_j(x_j, 1 - x_j, t)\right)\right|<1
	\ee
	for all $j\le n$,   $t\in\R$, and $1\le i\le l$,
	then any continuous  solution to the problem (\ref{lin})--\reff{per} 
	belongs to~$C_{per}^l(\Pi;\R^n)$.
\end{thm}

An analogous regularity statement for semilinear problems is formulated as follows.

\begin{thm}
	\label{thm:semilin}
	Let $l\ge 2$ be an arbitrary fixed integer and the condition 
	(\ref{aj}) be fulfilled. Let~$u$ be a classical $C^1(\Pi;\R^n)$ solution to 
	the problem \reff{semilin}, \reff{bc}, \reff{per}. Assume that,  for all $1\le j,k\le n$,  the coefficients
	$a_j$, $f_j$, $m_{jk}$, $h_j$, and $p_{jk}$ 
	belong to  $C^{l}_{per}$ with respect to all their arguments and on their respective domains.
	Moreover, assume that for all $j\le n$,   $t\in\R$, and $2\le i\le l$, it holds
		\beq\label{nonrez_semilin}
	\exp \left\{\int_{1 - x_j}^{x_j}
	\biggl(-\frac{F_{jj}}{a_{j}}-i\frac{\d_ta_j}{a_{j}^2}\biggr)(\eta,\om_j(\eta,1 - x_j,t))\dd\eta\right\}\sum_{k=1}^n
	\left|r_{jk}\left(\omega_j(x_j, 1 - x_j, t)\right)\right|<1,
\ee
where $x_j$ are defined  in \reff{*k} and
\beq\label{r_t}
\begin{array}{rcl}
r_{jk}(t)&=&\d_{k+1}	h_j\Bigl(t, u_{m+1}(0,t),\dots,u_n(0,t), u_1(1,t),\dots,u_m(1,t),[Pu](t)\Bigr).
\end{array}
\ee
	Then
	   $u$  belongs to
	$C^l_{per}(\Pi; \R^n)$.
\end{thm}

To formulate the corresponding   result in  the quasilinear setting, let   $u\in C^1_{per}(\Pi; \R^n)$ be given.
For fixed $(x,t)\in\Pi$ and $j\le n$,
we  will use the same notation  $\om_j(\xi)=\om_j(\xi,x,t)$ as above to denote
 the unique solution to the  Cauchy problem
\begin{equation}\label{char_q}
	\partial_\xi\omega_j(\xi, x,t)=\frac{1}{a_j(\xi,\omega_j(\xi),u(\xi,\om_j(\xi)))},\;\quad
	\omega_j(x,x,t)=t.
\end{equation}
Although $\om_j(\xi,x,t)$ depends also on $u$, we omit this dependence in the notation for brevity, since~$u$ is  fixed. 
\begin{thm}
		\label{thm:quasilin}
			Let $l\ge 3$ be an arbitrary fixed integer and let $u$ be a classical solution to 
		the problem \reff{quasilin}, \reff{bc}, \reff{per}, which belongs to $C^{2}_{per}(\Pi; \R^n)$. 
 Assume that,  for all $1\le j,k\le n$,  the coefficients 
$a_j$, $f_j$, $m_{jk}$, $h_j$, and $p_{jk}$ 
belong to  $C^{l}_{per}$
with respect to all their arguments and on their respective domains.
Moreover, assume that 
	\beq\label{aj_quasilin}
	a_j(x,t,u(x,t))\ne 0\ \mbox{ and }\ a_j(x,t,u(x,t))\ne a_k(x,t,u(x,t))\quad\mbox{ for all } (x,t)\in\Pi \mbox{ and } j\ne k,
	\ee
	and, for all  $j\le n$,   $t\in\R$, and $3\le i\le l$,
	it holds
	\begin{eqnarray}
\displaystyle
&\displaystyle\exp \left\{\int_{ 1 - x_j}^{x_j}
\biggl(\frac{F_jA_{jj}}{A_{j}^2}-\frac{F_{jj}}{A_{j}}-\frac{1}{A_j^2}\sum_{k=1}^n A_{jk}\d_tu_k-\frac{A_{jj}}{A_j^2}\d_tu_j-i\frac{A_{jt}}{A_{j}^2}\biggr)(\eta,\om_j\left(\eta,1 - x_j,t)\right)\dd\eta\right\}&
\nonumber\\
\displaystyle&\displaystyle\times\sum_{k=1}^n\left|r_{jk}\left(\omega_j(x_j, 1 - x_j, t)\right)\right|<1&\label{nonrez_quasilin}
	\end{eqnarray}
	with $r_{jk}$ as in \reff{r_t} and with $x_j$ as in \reff{*k}.
	Then
	$u$  belongs to
	$C^l_{per}(\Pi; \R^n)$.
	\end{thm}

\begin{rem}\rm
 The conditions \reff{nonrez}, \reff{nonrez_semilin}, and \reff{nonrez_quasilin} serve to exclude or control resonance mechanisms, specifically those associated with the loss of invertibility of the linearized or leading-order parts of the operators.
A detailed connection to small-divisor phenomena in Fourier representations will be discussed in Remark \ref{why_nonrez} below. For the role of small divisors and nonresonance conditions in the construction of periodic solutions to nonlinear wave equations, see, for example, \cite{CraigWayne}. 
\end{rem}

\begin{rem}(Relationship between Theorems \ref{thm:lin}, \ref{thm:semilin}, and \ref{thm:quasilin}).\label{Thmm} \rm
Assume that the semilinear system \reff{semilin} is linear, that is, the functions $f_j$ and $h_j$ are linear with respect to all arguments containing~$u$. This means that the equations \reff{semilin} and \reff{bc} take the form of \reff{lin} and \reff{bc_lin}, respectively.
Moreover, assume that $u$ is a classical $C^1$-solution to 
the problem \reff{semilin}, \reff{bc}, \reff{per}. Then,  for  given 
$l\ge 2$, the nonresonance conditions \reff{nonrez_semilin} for $2\le i\le l$ ensuring $C^l$-regularity of $u$ coincide with the nonresonance conditions 
\reff{nonrez} for $2\le i\le l$ established for the linear setting.
As follows from the proof of the linear Theorem~\ref{thm:lin}, the nonresonance condition \reff{nonrez} for $i=1$, which does not have an analog
in the semilinear setting (Theorem \ref{thm:semilin}), is crucial to prove
the $C^1$-regularity of continuous solutions to linear problems. This explains why this condition is absent in the semilinear setting, where one starts with $C^1$-solutions.

Similarly, if the quasilinear system \reff{quasilin} is semilinear, that is, $a_j$ are independent of  
$u$, and if $u$ is a classical $C^2$-solution to 
the problem \reff{quasilin}, \reff{bc}, \reff{per}, then for any 
$l\ge 3$ the nonresonance conditions \reff{nonrez_quasilin} for $3\le i\le l$ ensuring $C^l$-regularity of $u$ coincide with the nonresonance conditions 
\reff{nonrez_semilin} for $3\le i\le l$  established for the semilinear setting. Again, in the semilinear setting
the condition \reff{nonrez_semilin} for $l=2$, which ensures that a $C^1$-solution attains $C^2$-regularity, has no counterpart in the quasilinear setting, where one starts with $C^2$-solutions.

The regularity assumptions on the coefficients remain consistent in all three theorems.
\end{rem}

\begin{rem}(Dependence of nonresonance conditions on solutions). \label{dep_on_sol}\rm
Unlike the linear Theorem~\ref{thm:lin}, an important  feature of the nonlinear Theorems~\ref{thm:semilin} and~\ref{thm:quasilin} is that the nonresonance conditions \reff{nonrez_semilin} and~\reff{nonrez_quasilin} depend on the solution itself.
Furthermore, in establishing the $C^l$-regularity of solutions for arbitrary $l$, the condition \reff{nonrez_semilin} in the semilinear case requires only the continuity of the given $C^1$-solution, whereas the condition \reff{nonrez_quasilin} in the quasilinear case requires only the $C^1$-regularity of the given $C^2$-solution.
	
If a solution $u$ is sufficiently close to a stationary solution $u_0$, then the corresponding nonresonance conditions can often be reduced, by perturbation arguments, to conditions formulated solely in terms of $u_0$. This mechanism is illustrated by the results of \cite{KRT_evol}. Under suitable smallness assumptions on the nonlinearities $h_j$ and $f_j$, this paper establishes existence and uniqueness of small $C^2$-solutions, close to zero, for quasilinear problems of the form \reff{quasilin}, \reff{bc}, \reff{per}. The three nonresonance conditions imposed there  are of the form  \reff{nonrez_quasilin} with $i=0$, $i=1$, and $i=2$ evaluated at $u=0$ (see \cite[Theorem]{KRT_evol}). They are therefore independent of the particular solution $u\approx 0$.
	
Solution-dependent nonresonance conditions also arise in quasilinear
Kolmogorov-Arnold-Moser (KAM) and Nash-Moser theories. There, the
operators to be inverted are obtained by linearization about an
approximate solution, typically of small amplitude, so that their
coefficients depend on that solution. Consequently, the corresponding
Melnikov nonresonance conditions depend, at least implicitly, on the
underlying approximate solution (see, for example,
\cite{BertiBolle2020,BertiBiascoProcesi14}). In the present setting, this
dependence is explicit.  The nonresonance conditions are formulated
directly along the time-periodic solution itself.

\end{rem}

\begin{rem}(Contribution of integral terms). \label{integral_terms}\rm
It is worth noting that the integral terms $Mu$ and $Pu$, defined by \reff{J} and \reff{integral}, do not affect the nonresonance condition \reff{nonrez} in the linear setting. In contrast, in the nonlinear settings, if these terms appear in the nonlinear functions $f$ and $h$, they enter the corresponding nonresonance conditions \reff{nonrez_semilin} and \reff{nonrez_quasilin} through the quantities $F_j$ and $F_{jj}$.
\end{rem}

The paper is organized as follows. 
In Subsection~\ref{lit}, we motivate our approach and review related results.
Section~\ref{abstr} is devoted to an abstract regularity principle in vector
spaces, formulated in Lemma~\ref{lem:smoothing}. In Section~\ref{regularity}, this lemma is used as the main technical tool to
prove the higher-regularity results in Banach spaces of continuous functions
stated in Theorems~\ref{thm:lin}, \ref{thm:semilin}, and
\ref{thm:quasilin}. Section~\ref{sec:nonres} identifies 
a class of problems for which higher regularity follows without imposing any 
additional nonresonance conditions. Finally,
Section~\ref{nonrez-cond} provides a detailed discussion of the results, with
particular emphasis on the structure of the nonresonance conditions and their
role in higher regularity and in the distinction between autonomous and
nonautonomous problems.

\subsection{Motivation and related work}\label{lit}

The problem of higher regularity for time-periodic solutions of hyperbolic PDEs is closely connected with small-divisor phenomena and nonresonance conditions. KAM theory provides a natural point of comparison. Classical KAM results mainly concern small-amplitude periodic or quasiperiodic solutions of semilinear wave equations, while more recent developments also cover certain quasilinear models under additional structural assumptions (see, e.g., \cite{BertiBiascoProcesi14,CraigWayne93,Poschel96,Wayne90}). In KAM schemes, small divisors arise from the inversion of linearized operators occurring at successive steps of the iteration, and nonresonance conditions are imposed to control them.

No KAM iteration is used in the present paper. Instead, resonance phenomena appear as possible obstructions to higher regularity of time-periodic solutions. We identify nonresonance conditions ensuring higher regularity and show that, in general, they are essential. The underlying mechanism is already present for linear systems and extends naturally to semilinear and quasilinear problems.

Furthermore, the choice of function spaces plays an essential role in
the analysis of small-divisor effects and nonresonance conditions.
Similar issues arise, for example, in the study of periodic solutions of
nonlinear wave equations in Sobolev spaces \cite{BertiBolle2020}.

Higher regularity is also important in inverse and numerical problems for
hyperbolic equations. In inverse problems, smoothness of the forward map is
needed for stability, differentiability, and reconstruction analysis
\cite{Bertero98,inv-waves,Gerken2,Gerken1,Kirsch11}. Similarly, Newton-type and Bayesian methods rely on suitable stability
and regularity properties of parameter-to-solution maps in appropriate
function spaces (see, e.g., \cite{Dashti17}).

The class of problems considered here is  motivated by applications.
Nonautonomous linear and nonlinear hyperbolic (integro-differential)
equations of the type described in Section~\ref{results} arise in
structured population dynamics \cite{Ianelli,Inaba,Webb}, correlated random walk
models for chemotaxis \cite{Hillen}, models of semiconductor laser dynamics
\cite{LRR,RW}, and boundary control theory \cite{Coron_Fr,Pavel}. For
well-posedness and Fredholm properties of such problems, we refer to
\cite{Coron_Fr,KR_IFT}. In many of these applications, periodic or
almost periodic regimes naturally appear due to external forcing or
intrinsic oscillatory mechanisms.

In the autonomous linear and semilinear cases, regularity
questions for one-dimensional hyperbolic systems of the type considered here
were studied, for example, in \cite{Hale} for bounded solutions near a constant,
and in our previous works
\cite{KR_IFT,KR_regul_ex,KR_autonom,KRT_evol}. Related higher-regularity issues also arise in multiscale models for transport
in porous media. For example, in \cite{MelnykPopRohde26} the authors study a nonlinear
reactive-transport problem in a fractured porous medium whose homogenized limit
contains a first-order one-dimensional semilinear hyperbolic system. 
There, the $C^3$-regularity of the solution is needed,
in order to prove quantitative  error estimates.

Higher regularity of the solution operator is also important for bifurcation
analysis. In particular, Lyapunov--Schmidt reduction performed directly in
spaces of periodic functions requires suitable smoothness of the solution
operators (see, e.g., \cite{Kielhofer,KR_Hopf,KR_IFT}). The present paper
therefore fits into a broader program of understanding when time-periodic
solutions of hyperbolic problems possess additional smoothness, and which
conditions are indeed necessary for this to hold.

For {\it initial and initial-boundary} value problems, classical persistence results were obtained by Friedrichs in \cite{Friedrichs}, showing that solutions of one-dimensional first-order strictly hyperbolic systems inherit the regularity of the coefficients, source terms, boundary data, and initial data. In the presence of boundary conditions,
higher regularity additionally requires compatibility conditions of the
corresponding order between the initial and boundary data
\cite{RM74}.

The situation is different for {\it time-periodic} problems. In this setting,
compatibility conditions do not arise, but regularity of the data alone is
not, in general, sufficient to ensure regularity of time-periodic solutions.
This obstruction occurs not only in nonlinear hyperbolic equations but already
in linear  systems. We show that higher regularity of the coefficients and boundary data does not, in general, imply higher regularity of time-periodic solutions unless suitable nonresonance conditions are imposed. In the autonomous case, a single nonresonance condition suffices for arbitrarily high regularity and is, in general, essential \cite[Remark~2.3]{KR_autonom}. In the nonautonomous case, by contrast, the number of required nonresonance conditions is tied to the desired order of regularity. This phenomenon is illustrated by the examples in \cite[Remark~1.4]{KR_regul_ex} and \cite[Section~3.6]{KRT_evol} (see also Subsection~\ref{number} below).

The classical approach to higher regularity for hyperbolic systems, described in \cite{Friedrichs} (see also \cite{Wloka}), is based on differentiating the system formally with respect to time, up to the order allowed by the coefficients. One then solves the resulting auxiliary systems for
$\d_tu$, $\d_t^2u$, $\dots$, $\d_t^ku$ up to the desired order and finally verifies that these auxiliary unknowns are indeed the corresponding derivatives of the original solution $u$. Thus, the proof reduces higher regularity to the solvability and regularity properties of the corresponding auxiliary problems.

In contrast, one of the main ingredients of the present paper is an abstract
regularity principle, formulated purely in terms of vector spaces and proved in
Lemma~\ref{lem:smoothing}. Rather than introducing and solving additional differentiated systems, the lemma exploits improved mapping properties of the operators involved. When applied to the hyperbolic problems studied here, its first condition gives rise to the relevant nonresonance requirement. Thus, the abstract lemma identifies the operator conditions responsible for regularity improvement, while the specific analysis of the hyperbolic problems verifies these conditions.

\section{An abstract regularity principle in vector spaces}\label{abstr}
\renewcommand{\theequation}{{\thesection}.\arabic{equation}}
\setcounter{equation}{0}

The main result of this section is an abstract lemma formulated in a vector-space setting. 
\begin{lem}\label{lem:smoothing}  
	Let $V$ be a vector space, $U\subseteq V$ be a subspace of $V$, and $\CC,\E: V\to V$ be linear operators. 
Assume that integers $m\ge 1$ and $k\ge 2$ are such that:
	\begin{enumerate}
		\item 
	If $(I-\CC^m)v\in U$  for $v\in V$, 
		then  $v\in U$;
\item 
$\CC u\in U$ and $\E u\in U$ for all $u\in U$;
		\item 
		$\sgn(m-1)\,\CC\E u$, $\E\CC u$, and $\E^ku$ belong to $U$ for all $u\in V$.
	\end{enumerate}
Then for any $u\in V$  satisfying $u-\CC u-\E u\in U$, it follows that  $u\in U$.
\end{lem}

\begin{proof}
	Take $u\in V$ satisfying  $u-\CC u-\E u\in U$, and set $v=u-\CC u-\E u$. Hence, $v\in U$.
	Then, by performing suitable iterations, we obtain
	\beq\label{u-Cu}
	\begin{array}{rcl}
		u-\CC u&=&\E u+v=\E (\E u+\CC u+v)+v\\
		&=&\E^2u+ \E\CC u+(I+\E)v\\
		&=& \E^2(\E u+\CC u+v)+ \E\CC u+(I+\E)v\\
		&=&\E^3u+(I+\E)\E\CC u+(I+\E+\E^2)v=\dots\\
		&=& \E^ku+(I+\E+\dots+\E^{k-2})\E\CC u+(I+\E+\dots+\E^{k-1})v.
	\end{array}
	\ee
	This proves the statement for $m=1$. Indeed, by Conditions~2 and~3, the right-hand side of \reff{u-Cu} belongs to $U$, and therefore so does the left-hand side. Hence, Condition~1 implies that $u\in U$.

To establish the statement for $m>1$, we start from \reff{u-Cu} and perform further iterations, as follows:
$$
\begin{array}{ll}
	u-\E^ku-(I+\E+\dots+\E^{k-2})\E\CC u-(I+\E+\dots+\E^{k-1})v=\CC u\\
\quad\qquad	=\ \,\CC(\E u+\CC u+v)=\CC^2u+\CC\E u+\CC v\\
\quad\qquad	=\ \,\CC^2(\E u+\CC u+v)+\CC\E u+\CC v\\
\quad\qquad=\ \,\CC^3u+(I+\CC)\CC\E u+(I+\CC)\CC v=\dots\\
\quad\qquad=\ \,\CC^mu+(I+\CC+\dots+\CC^{m-2})\CC\E u+(I+\CC+\dots+\CC^{m-2})\CC v.
	\end{array}
$$
It follows that
	\beq\label{u-Clu}
	\begin{array}{rcl}
		u-\CC^mu &=&\E^ku+(I+\E+\dots+\E^{k-2})\E\CC u+(I+\CC+\dots+\CC^{m-2})\CC\E u\\
		&&+(I+\E+\dots+\E^{k-1})v+(I+\CC+\dots+\CC^{m-2})\CC v.
	\end{array}
\ee
Again, by Conditions~2 and~3, the right-hand side of \reff{u-Clu} belongs to $U$, and hence the left-hand side also  belongs to $U$. Consequently, Condition~1 implies that $u\in U$.
\end{proof}

\begin{rem}(About Condition 1 and the role of the parameter $m$).\label{suf} \rm
	Note that Condition~1 is automatically satisfied whenever the operator $I-\CC^m$ is bijective from $U$ onto $U$ for some $m\ge1$. The converse implication, however, does not hold in general. This can be seen from the proof of Claim~\ref{I-C} in Section~\ref{1} for the time-periodic linear hyperbolic problems considered here. The same proof also shows that the conditions \reff{nonrez} (and, similarly, \reff{nonrez_semilin} and \reff{nonrez_quasilin}) are sufficient for Condition~1 with~$m=1$.
	
	Deriving analogous sufficient conditions for higher values of $m>1$ yields additional sufficient nonresonance conditions beyond those obtained for $m=1$.
	 For example, the conditions $R^1<1$ and $S^1<1$ in \cite[Theorem~1.2]{KR_regul_ex}) guarantee $C^1$-regularity of continuous solutions to linear problems and are sufficient for Condition 1 with $m=2$.
	
It turns out that, unlike the explicit nonresonance conditions
\reff{nonrez}, \reff{nonrez_semilin}, and \reff{nonrez_quasilin}, which are
only sufficient, the family of conditions arising from Condition~1 of
Lemma~\ref{lem:smoothing} for all $m\geq1$ captures the essential
nonresonance mechanism governing higher regularity. More precisely, the
parameter $m$ provides a hierarchy of increasingly refined criteria, and
the union over all $m\geq1$ of the corresponding classes of systems gives
a natural approximation to the maximal nonresonant region detected by this
criterion. The complementary regime, where Condition~1 fails for every
$m\geq1$, is naturally associated with the so-called  {\it completely resonant behavior}.
 It may be viewed as the small-divisor regime, in which the smoothing mechanism is absent and small divisors are expected to occur for generic parameter values.

A particularly transparent case is provided by nilpotent 
operators $\CC$. Indeed, if 
$
\mathcal C^m=0
$
for some $m\geq1$, then Condition~1 is automatically satisfied for this value of $m$. Hence, every system with nilpotent $\mathcal C$ belongs to one of the nonresonant regions corresponding to
$m\geq1$.
For the PDE problems considered here, we refer to the resulting distinguished class of time-periodic hyperbolic problems as {\it completely nonresonant} (see Section~\ref{sec:nonres}).
\end{rem}

\begin{rem} (About the role of the parameter $k$). \rm
As it follows from the proof of Lemma~\ref{lem:smoothing}, the parameter $k$  determines how many iterations of the ``regularizing mechanism'' are required in order to be able to transfer a solution from the space $V$ to its  (higher-regularity) subspace~$U$. As will be shown in Section \ref{proofs}, for the one-dimensional hyperbolic problems considered here, the choice $k=2$ is sufficient. In multidimensional hyperbolic problems, however, larger values of
	$k$ may be necessary in general. In particular, for certain classes of
	$r$-dimensional first-order hyperbolic systems (including those for which
	the classical method of characteristics applies), one may require
	$k=r+1$ iterations. This result will be established in a forthcoming
	paper.
\end{rem}	



\section{Higher regularity of time-periodic solutions\label{proofs}
}
\label{regularity}
\renewcommand{\theequation}{{\thesection}.\arabic{equation}}
\setcounter{equation}{0}

\subsection{Linear hyperbolic systems: proof of  Theorem \ref{thm:lin}}\label{1}

Assume that all conditions of Theorem \ref{thm:lin} are fulfilled, and let $u$ be a continuous solution to 
the problem (\ref{lin})--\reff{per} in the sense of Definition \ref{cont}.
Our strategy is to rewrite the problem (\ref{lin})--\reff{per} in the abstract 
operator form \reff{oper} below, and then  repeatedly apply Lemma \ref{lem:smoothing}  to it.  To this end, 
for $i\in\N_0$ and $j\le n$, write
\begin{equation} \label{cd}
	\hspace{-2mm}	c_j^i(\xi,x,t)=\exp \int_x^\xi
	\left[\frac{b_{jj}}{a_{j}} - i\frac{\partial_t a_{j}}{a_{j}^2} \right](\eta,\omega_j(\eta,x,t))\dd\eta,\qquad 
	d_j^i(\xi,x,t)=\frac{c_j^i(\xi,x,t)}{a_j(\xi,\omega_j(\xi,x,t))},
\end{equation}
where $\om_j$ are the characteristic curves of (\ref{lin}), determined as solutions to \reff{char}.
We will simply write $c_j(\xi,x,t)$ and $d_j(\xi,x,t)$ for $c_j^0(\xi,x,t)$ and $d_j^0(\xi,x,t)$, respectively.

Introduce   operators    
$\H\in {\cal L}(C_{per}(\R;\R^n))$,  $\F\in {\cal L}\left(C_{per}(\Pi;\R^n)\times C_{per}(\Pi;\R^n),C_{per}(\Pi;\R^n)\right)$, and $\CC, Q, B\in {\cal L}(C_{per}(\Pi;\R^n))$ by
\begin{equation}\label{CDF}
	\begin{array}{rcl}
		[\CC u]_j(x,t)&=& c_j(x_j,x,t)[Ru]_j(\omega_j(x_j,x,t)),
		\\[2mm]
			[\H h]_j(x,t)&=& c_j(x_j,x,t)h_j(\omega_j(x_j,x,t)),
		\\[2mm]
	[Qu]_j(x,t)&=&\displaystyle   c_j(x_j,x,t)\sum_{j=1}^n\int_0^1p_{jk}(\eta,\omega_j(x_j,x,t))u_k(\eta,\omega_j(x_j,x,t))\dd \eta,	
		\\[2mm]		
		\displaystyle
[Bu]_j(x,t)&=&\displaystyle
		-\int_{x_j}^{x}  d_j(\xi,x,t)\sum_{k\neq j}  b_{jk}(\xi, \omega_j(\xi,x,t))u_k(\xi, \omega_j(\xi,x,t)) \dd\xi,\\ [2mm]
		\left[\F (f,u)\right]_j(x,t)&=&\displaystyle\int_{x_j}^{x}d_j(\xi,x,t)\left(f_j(\xi, \omega_j(\xi,x,t))- \left[Mu\right]_j(\xi,\omega_j(\xi,x,t))\right)\dd\xi.
	\end{array}
\end{equation}
After integration  along the characteristic curves, the problem (\ref{lin})--\reff{per} 
takes the operator form
\begin{eqnarray} \label{oper}
	u=\CC u+Qu+Bu+\H h+\F(f,u),
\end{eqnarray}  
where  $f=(f_1,\dots,f_n)$ and $h=(h_1,\dots,h_n)$.

For $i\in\N_0$, we also introduce operators $G_i\in {\cal L}(C_{per}(\R, \R^n))$ by
\begin{equation} \label{Ci}
	[G_i\psi]_j(t) = c_j^i(x_j, 1 - x_j, t)
	\sum\limits_{k=1}^nr_{jk}\left(\omega_j(x_j, 1 - x_j, t)\right)\psi_k\left(\omega_j(x_j, 1 - x_j, t)\right)\quad 
	\mbox{for all } j \le n.
\end{equation}
Because of assumption \reff{nonrez} of Theorem \ref{thm:lin}, we have that
\begin{equation}
	\label{invert}
	\|G_i\|_{{\cal L}(C_{per}(\R;\R^n))}<1 \quad\mbox{ for all } i\le n.
\end{equation}

On the first step we show that  the continuous solution $u$ belongs to $C^1_{per}(\Pi;\R^n)$ and, hence, is a classical solution   to  (\ref{lin})--\reff{per}.
To this end, it is sufficient to   prove that the equation \reff{oper} fulfills all conditions of  Lemma \ref{lem:smoothing}
with $\E=Q+B$, $U=C^1_{per}(\Pi;\R^n)$, $V=C_{per}(\Pi;\R^n)$, and with $m=1$ and $k=2$.
The proof  goes through Claims \ref{I-C}--\ref{C1} below. In Claim \ref{I-C} we  verify Condition 1 of Lemma \ref{lem:smoothing}.

\begin{claim} \label{I-C}
If for  $v\in C_{per}(\Pi;\R^n)$ it holds $(I-\CC)v\in C^1_{per}(\Pi;\R^n)$, then  $v\in C^1_{per}(\Pi;\R^n)$.
\end{claim}
\begin{subproof}
According to the assumptions of the claim, 
for given $v\in C_{per}(\Pi;\R^n)$ satisfying $(I-\CC)v\in C^1_{per}(\Pi;\R^n)$,  
there exists $g\in C^1_{per}(\Pi;\R^n)$
such that  
\begin{equation}\label{simpl}
	v=\CC v+g.
\end{equation}
 Set
\beq\label{z}
\begin{array}{rcl}
z(t) &=& (z_1(t),\dots,z_n(t))= (v_1(1,t),\dots, v_m(1,t), v_{m+1}(0,t),\dots,v_n(0,t)),\\ [2mm]
\tilde g(t)&=&\left(\tilde g_1(t),\dots,\tilde g_n(t)\right)=\left(g_1(1,t),\dots,g_m(1,t),g_{m+1}(0,t),\dots,g_n(0,t)\right).
\end{array}
\ee	
Hence, $z\in C_{per}(\R;\R^n)$ and $\tilde g\in C_{per}^1(\R;\R^n)$.
As follows from~\reff{simpl} and from the regularity assumptions on the data in Theorem \ref{thm:lin}, the proof will be complete once we  show that $z\in C^1_{per}(\R;\R^n)$. To this end, 
 consider 	(\ref{simpl}) 	at $x=0$ for $m<j\le n$ and at $x=1$ for $1\le j\le m$. 
In the notations \reff{Ci} and  \reff{z},
the resulting equation obtained in this way from \reff{simpl} then reads
\begin{equation}\label{simpl1}
	z=G_0z+ \tilde g.
\end{equation}
If  $y\in C^1_{per}(\R;\R^n)$, then due to \reff{Ci}, it holds
	\begin{equation}\label{dtG0R}
	(G_0 y)^\prime = G_1y^\prime+ W_1 y,
\end{equation}
	where 
	the operator $W \in \mathcal L(C_{per}(\R; \R^n))$ is defined by
\beq\label{W}
[W_1 w]_j(t) =  	 \sum\limits_{k=1}^n\frac{\dd}{\dd t}\Bigl[c_j(x_j,1 - x_j,t)\,r_{jk}(\omega_j(x_j, 1 - x_j, t))\Bigr]w_k(\omega_j(x_j, 1 - x_j, t)),
	\quad j \le n.
	\ee
	
Since $C^1_{per}(\R;\R^n)$	is dense in $C_{per}(\R;\R^n)$, choose    a sequence $(z^l)\subset C^1_{per}(\R;\R^n)$  such that
	  $z^l\to z$ in $C_{per}(\R;\R^n)$ as $l\to\infty$. Moreover, 
	  for every $\varphi\in \left(C_{per}(\R;\R^n)\right)^*$,
 there exists a sequence
$
(\varphi^k)\subset
C^\infty_{\mathrm{per}}(\mathbb{R};\mathbb{R}^n)
$
such that
\beq\label{ap}
\langle\varphi,y\rangle
=
\lim_{k\to\infty}
\int_0^{2\pi} \varphi^k(t)\cdot y(t)\dd t
=
\lim_{k\to\infty}
\langle\varphi^k,y\rangle
\ee
for every
$
y\in C_{\mathrm{per}}(\mathbb{R};\mathbb{R}^n).
$
Here $\cdot$ denotes the Euclidean scalar product in $\R^n$. 

Further, because of \reff{invert}, there exists $v\in C_{per}(\R;\R^n)$ such that	
$$
(I-G_1)v=W_1z+\tilde g'.
$$

Now,  take $\varphi\in \left(C_{per}(\R;\R^n)\right)^*$ and $
(\varphi^k)\subset
C^\infty_{\mathrm{per}}(\mathbb{R};\mathbb{R}^n)
$
with the property \reff{ap}. Then, by \reff{simpl1} and \reff{dtG0R},  for every $k\ge 1$, it holds
	\begin{eqnarray*}
	\lim_{l\to\infty} \left\langle\vphi^k, (I-G_1)(z^l)^\prime\right\rangle&=& -\lim_{l\to\infty} \left\langle (\vphi^k)^\prime,z^l
		\right\rangle -
		\lim_{l\to\infty} \left\langle \vphi^k,G_1(z^l)^\prime\right\rangle
	\\
	&	=& -\left\langle (\vphi^k)^\prime,z\right\rangle -
		\lim_{l\to\infty} \left\langle \vphi^k,G_1(z^l)^\prime\right\rangle
	\\
	&	
		= &-\left\langle (\vphi^k)^\prime,G_0z + \tilde g\right\rangle -
		\lim_{l\to\infty} \left\langle \vphi^k,G_1(z^l)^\prime\right\rangle
	 	\\
	 &	=& -\lim_{l\to\infty} \left\langle (\vphi^k)^\prime,G_0z^l + \tilde g\right\rangle -
	 \lim_{l\to\infty} \left\langle \vphi^k,G_1(z^l)^\prime\right\rangle
	\\
	&	 =&
		\lim_{l\to\infty} \left\langle \vphi^k,G_1(z^l)^\prime+W_1z^l + \tilde g^\prime-G_1(z^l)^\prime\right\rangle-\lim_{l\to\infty} \left\langle \vphi^k,G_1(z^l)^\prime\right\rangle	\\
		&	 =&\left\langle \vphi^k,W_1z+\tilde g'\right\rangle	=\left\langle \vphi^k,(I-G_1)v\right\rangle,
\end{eqnarray*}	  
where $\langle\cdot,\cdot\rangle: \left(C_{per}(\R;\R^n)\right)^*\times C_{per}(\R;\R^n)$ denotes the duality pairing.
Hence, the limit on the left-hand side exists for every $k\ge 1$. Moreover, the limit as $k\to\infty$ on the right-hand side (and, hence, on the left-hand side) exists as well. Therefore, passing to the limit as $k\to\infty$, we conclude that $(I-G_1)(z^l)^\prime$ tends weakly in $C_{per}(\R;\R^n)$ to $(I-G_1)v$ as $l \to \infty$.
Therefore,
\beq\label{w}
(z^l)^\prime\rightharpoonup v\quad \mbox{in } C_{per}(\R;\R^n)
\mbox{ as } l \to \infty.
\ee

For given $t\in\R$
define the functional
$J_{t}\in \left(C_{per}(\R;\R^n)\right)^*,$ by
$\langle J_{t},w\rangle
:=
\int_0^t w(\tau)\dd \tau$.
Then, by the weak convergence \reff{w} and
by the fundamental theorem of calculus,
$$
z(t)-z(0)=\lim_{l\to \infty}(z^l(t)-z^l(0))
=\lim_{l\to \infty}
\int_0^t (z^l)'(\tau)\dd\tau
=\lim_{l\to \infty}
\langle J_{t},(z^l)'\rangle
=\langle J_{t},v\rangle=\int_0^t v(\tau)\dd\tau.
$$
It follows that
$z\in C^1(\R;\R^n)$.
The proof of the claim is complete. 
	\end{subproof}

Next,  we  verify Condition 3 (with $m=1$) of Lemma \ref{lem:smoothing} by showing that 	the functions  $\E\CC u$,
 and $\E^2u$  belong to  $C^1_{per}(\Pi;\R^n)$.
Since $\E=Q+B$, we are done if we prove the following statement.
		
\begin{claim} \label{B2}
	The functions $Q\CC u, B\CC u, Q^2u, B^2u, QBu$, and $BQu$ belong to $C^1_{per}(\Pi;\R^n)$. 
\end{claim}
\begin{subproof}		
We start with  $Q^2u$. 
Since $Q$ is a sum of 
operators  	$P_{jk}\in\mathcal{L}(C_{per}(\Pi))$ given by
$$
\left[P_{jk}v\right](x,t)=\displaystyle   c_j(x_j,x,t)\int_0^1p_{jk}(\eta,\omega_j(x_j,x,t))v_k(\eta,\omega_j(x_j,x,t))\dd \eta,	
$$
it  suffices to prove that the function
$$
\begin{array}{rcl}
\left[P_{jk}P_{ki}u\right](x,t)&=&\displaystyle c_j(x_j,x,t)\int_0^1p_{jk}(\eta,\omega_j(x_j))c_k(x_k,\eta,\omega_j(x_j))\\
&&\displaystyle\times\int_0^1\left[p_{ki}u_i\right](\rho,\om_k(x_k,\eta,\omega_j(x_j))\dd\rho\dd \eta
\end{array}
	$$ 
belongs to $C^1_{per}(\Pi;\R^n)$	for any $j\le n$, $k\le n$, and $i\ne k$.

Fix arbitrary $j\le n$, $k\le n$, and $i\ne k$. Due to the  equalities
\beq\label{dx}
\begin{array}{rcl}
\partial_x\omega_j(\xi,x,t) & = & \displaystyle-\frac{1}{a_j(x,t)} \exp \int_\xi^x 
\frac{\partial_ta_j(\eta,\omega_j(\eta))}{a_j(\eta,\omega_j(\eta))^2}\dd\eta,\\ [3mm]
\partial_t\omega_j(\xi,x,t) & = & \displaystyle\exp \int_\xi^x 
\frac{\partial_ta_j(\eta,\omega_j(\eta))}{a_j(\eta,\omega_j(\eta))^2}\dd\eta,\\ [3mm]
\left(\d_t+a_j(x,t)\d_x\right)\om_j(x_j,x,t)&=&0,\\ [3mm]
\left(\d_t+a_j(x,t)\d_x\right)c_j(x_j,x,t)&=&-b_{jj}(x,t)c_j(x_j,x,t),
\end{array}
\ee
the following formula is true for any $v\in C^1_{per}(\Pi;\R^n)$:
\beq\label{directional}
\left(\d_t+a_j(x,t)\d_x\right)\left[P_{jk}P_{ki}v\right](x,t)=-b_{jj}(x,t)\left[P_{jk}P_{ki}v\right](x,t).
\ee

In view of \reff{directional} and the density of $C^1_{per}(\Pi;\R^n)$  in $C_{per}(\Pi;\R^n)$,  it suffices to show that, for every sequence
$(u^l)\subset C^1_{per}(\Pi;\R^n)$ satisfying  $u^l\to u$ in $C_{per}(\Pi;\R^n)$, the sequence$\left(\partial_t[P_{jk}P_{ki}u^l]\right)_{l\ge1}$ converges in $C_{per}(\Pi)$. Let $(u^l)\subset C^1_{per}(\Pi;\R^n)$ be such a sequence. We then compute
$$
	\begin{array}{ll}
	\d_t\left[P_{jk}P_{ki}u^l\right](x,t)=
\displaystyle \int_0^1\int_0^1u^l_{i}(\rho,\om_k(x_k,\eta,\omega_j(x_j))
\\ [3mm]\qquad\quad\times\displaystyle
\frac{\dd}{\dd t}\Bigl[c_j(x_j,x,t)p_{jk}(\eta,\omega_j(x_j))c_k(x_k,\eta,\omega_j(x_j))
p_{ki}(\rho,\om_k(x_k,\eta,\omega_j(x_j))\Bigr]\dd\rho\dd \eta	
\\ [3mm]\qquad\displaystyle
+c_j(x_j,x,t)\int_0^1p_{jk}(\eta,\omega_j(x_j))c_k(x_k,\eta,\omega_j(x_j))
\\ [3mm]\qquad\quad\times\displaystyle
\int_0^1p_{ki}(\rho,\om_k(x_k,\eta,\omega_j(x_j)))\,
\d_2u^l_i(\rho,\om_k(x_k,\eta,\omega_j(x_j)))\d_3\om_k(x_k,\eta,\omega_j(x_j)))\d_t\omega_j(x_j)\dd\rho\dd \eta.
\end{array}
$$ 	
We transform the second summand on the right-hand side, as follows. Calculating the derivative
\begin{eqnarray}
		\frac{\dd}{\dd\eta}u_i^l(\rho,\om_k(x_k,\eta,\omega_j(x_j)))= \d_2u_i(\rho,\om_k(x_k,\eta,\omega_j(x_j)))\,\d_2\om_k(x_k,\eta,\omega_j(x_j))).
\end{eqnarray}
and combining it with the equality (which is true due to \reff{dx})
\beq\label{om23}
\d_3\om_k(\xi,\eta,\omega_j(x_j))=-a_j(\eta,\omega_j(x_j))\,\d_2\om_k(\xi,\eta,\omega_j(x_j)),
\ee
we derive the formula
$$
\d_2u^l_i(\rho,\om_k(x_k,\eta,\omega_j(x_j)))\d_3\om_k(x_k,\eta,\omega_j(x_j)))=-a_j(\eta,\omega_j(x_j))\frac{\dd}{\dd\eta}u^l_i(\rho,\om_k(x_k,\eta,\omega_j(x_j))).
$$
This leads to the following representation formula for
$\d_t\left[P_{jk}P_{ki}u^l\right](x,t)$:
\beq\label{PP}
\begin{array}{ll}
	\d_t\left[P_{jk}P_{ki}u^l\right](x,t)=
	\displaystyle \int_0^1\int_0^1u^l_{i}(\rho,\om_k(x_k,\eta,\omega_j(x_j))
	\\ [3mm]\qquad\quad\times\displaystyle
	\frac{\dd}{\dd t}\Bigl[c_j(x_j,x,t)p_{jk}(\eta,\omega_j(x_j))c_k(x_k,\eta,\omega_j(x_j))
	p_{ki}(\rho,\om_k(x_k,\eta,\omega_j(x_j))\Bigr]\dd\rho\dd \eta	
	\\ [3mm]\qquad\displaystyle
	-c_j(x_j,x,t)\d_t\omega_j(x_j)\int_0^1\Bigl[p_{jk}(\eta,\omega_j(x_j))c_k(x_k,\eta,\omega_j(x_j))
	\\ [3mm]\qquad\quad\times\displaystyle
	p_{ki}(\rho,\om_k(x_k,\eta,\omega_j(x_j)))\,
	a_j(\eta,\omega_j(x_j))u^l_i(\rho,\om_k(x_k,\eta,\omega_j(x_j)))\Bigr]_{\eta=0}^{\eta=1}\dd\rho
\\ [3mm]\qquad\displaystyle	+c_j(x_j,x,t)\d_t\omega_j(x_j)\int_0^1\int_0^1
\frac{\dd}{\dd\eta}\Bigl[
p_{jk}(\eta,\omega_j(x_j))c_k(x_k,\eta,\omega_j(x_j))
\\ [3mm]\qquad\quad\times\displaystyle
p_{ki}(\rho,\om_k(x_k,\eta,\omega_j(x_j)))\,
a_j(\eta,\omega_j(x_j))
\Bigr]u^l_i(\rho,\om_k(x_k,\eta,\omega_j(x_j)))\dd \eta\dd\rho.
\end{array}
\ee
Using the regularity and periodicity of the coefficients of the original problem, together with uniform convergence of $u^l\to u$ in $C_{per}(\Pi)$, we conclude that $\d_t\left[P_{jk}P_{ki}u^l\right]$
converges uniformly to $\d_t\left[P_{jk}P_{ki}u\right]$ in $C_{per}(\Pi)$.
This establishes the desired $ C^1_{per}(\Pi;\R^n)$-regularity of $Q^2u$.

	Next, we give the proof for the function $B^2u$.
To this end,  for fixed $j\le n$, we consider the representation of
$\left[B^2u\right]_j(x,t)$ obtained after changing the order of
integration:
	\begin{equation}\label{D11}
		\begin{array}{rcl}
			\left[B^2u\right]_j(x,t)
			=\displaystyle\sum_{k\not=j}\sum_{i\not=k}
			\int_{x_j}^x \int_\eta^x d_{jki}(\xi,\eta,x,t)\,
			u_i(\eta,\omega_k(\eta,\xi,\omega_j(\xi))) \dd \xi \dd \eta,
			\nonumber
		\end{array}
	\end{equation}
	where
	\begin{eqnarray*}
		d_{jki}(\xi,\eta,x,t)
		=d_j(\xi,x,t)\,d_k(\eta,\xi,\omega_j(\xi))\,b_{jk}(\xi,\omega_j(\xi))\,b_{ki}(\eta,\omega_k(\eta,\xi,\omega_j(\xi))).
	\end{eqnarray*}
	Because of \reff{cd} and \reff{dx}, we have
\beq\label{dt}
\begin{array}{rcl}
	\left(\d_t+a_j(x,t)\d_x\right)d_j(x_j,x,t)&=&-b_{jj}(x,t)d_j(x_j,x,t),\\ [3mm]
	\left(\d_t+a_j(x,t)\d_x\right)d_{jki}(x_j,x,t)&=&-b_{jj}(x,t)d_{jki}(x_j,x,t).
\end{array}
\ee	
Hence, for fixed  $j\le n$ and  $v\in C^1_{per}(\Pi;\R^n)$, it holds
\beq\label{directional1}
\left(\d_t+a_j(x,t)\d_x\right)\left[B^2v\right]_j(x,t)=-b_{jj}(x,t)\left[B^2v\right]_j(x,t).
\ee	

Similarly to the proof for the function $Q^2u$, in view of
\reff{directional1}, 
the proof for $B^2u$ is complete once we show that $\d_t\left[B^2u^l\right]_j$ converges in $C_{per}(\Pi)$ for every sequence $(u^l)\subset C^1_{per}(\Pi;\R^n)$
converging to $u$ in $C_{per}(\Pi;\R^n)$. To this end,  for fixed $j\le n$ and $v\in C^1_{per}(\Pi;\R^n)$, we  first compute the derivative 
\beq\label{dtB2}
\begin{array}{ll}
		\d_t\left[B^2v\right]_j(x,t)
	=\displaystyle\sum_{k\neq j}\sum_{i\neq k}\int_{x_j}^x\int_{\eta}^{x} \frac{\dd}{\dd t} d_{jki}(\xi,\eta,x,t)\, v_i(\eta,\omega_k(\eta,\xi,\omega_j(\xi))) \dd\xi \dd\eta
	\nonumber\\ [5mm]\quad
	+\displaystyle\sum_{k\neq j}\sum_{i\neq k}\int_{x_j}^x\int_{\eta}^{x} d_{jki}(\xi,\eta,x,t) \,
	\d_t\omega_k(\eta,\xi,\omega_j(\xi))\,\d_t\omega_j(\xi)\,\d_2v_i(\eta,\omega_k(\eta,\xi,\omega_j(\xi))) \dd\xi \dd\eta. 
\end{array}
\ee
We transform the second term on the right-hand side by using \reff{char} and \reff{dx} as follows:
\begin{eqnarray}
	\lefteqn{
		\frac{\dd}{\dd\xi} v_i(\eta,\omega_k(\eta,\xi,\omega_j(\xi)))} \nonumber \\ &&
	=\Bigl[\d_x\omega_k(\eta,\xi,\omega_j(\xi))+\d_t\omega_k(\eta,\xi,\omega_j(\xi))\d_{\xi}\omega_j(\xi)\Bigr] \d_2v_i(\eta,\omega_k(\eta,\xi,\omega_j(\xi))) \label{eqwn}
	\\ &&
	=\left ( \frac{1}{a_j(\xi,\omega_j(\xi))}-\frac{1}{a_k(\xi,\omega_j(\xi))}\right ) \d_t\omega_k(\eta,\xi,\omega_j(\xi))\,\d_2v_i(\eta,\omega_k(\eta,\xi,\omega_j(\xi))). \nonumber
\end{eqnarray}
Hence, 
$$
	\d_t\omega_k(\eta,\xi,\omega_j(\xi))\,\d_2v_i(\eta,\omega_k(\eta,\xi,\omega_j(\xi)))
	=\displaystyle\frac{a_j(\xi,\omega_j(\xi))a_k(\xi,\omega_j(\xi))}{a_k(\xi,\omega_j(\xi))-a_j(\xi,\omega_j(\xi))} \frac{\dd}{\dd\xi} v_i(\eta,\omega_k(\eta,\xi,\omega_j(\xi))).
$$
Set
$$
\tilde{d}_{jki}(\xi,\eta,x,t)
=d_{jki}(\xi,\eta,x,t)\d_t\omega_j(\xi)\frac{a_j(\xi,\omega_j(\xi))a_k(\xi,\omega_j(\xi))}{a_k(\xi,\omega_j(\xi))-a_j(\xi,\omega_j(\xi))}.
$$
and rewrite \reff{dtB2} as follows:
\begin{eqnarray*}
	\lefteqn{
		\d_t[(B^2v)_j(x,t)]} \\ &&
	= \displaystyle \sum_{k\neq j}\sum_{i\neq k}\int_{x_j}^x\int_{\eta}^{x} \frac{\dd}{\dd t} d_{jki}(\xi,\eta,x,t)\, v_i(\eta,\omega_k(\eta,\xi,\omega_j(\xi))) \dd\xi \dd\eta
	\\ &&\quad
	+\displaystyle \sum_{k\neq j}\sum_{i\neq k}\int_{x_j}^x\int_{\eta}^{x}\tilde{d}_{jki}(\xi,\eta,x,t)\frac{\dd}{\dd\xi} v_i(\eta,\omega_k(\eta,\xi,\omega_j(\xi))) \dd\xi \dd\eta
	\\ &&
	=\displaystyle \sum_{k\neq j}\sum_{i\neq k}\int_{x_j}^x\int_{\eta}^{x} \frac{\dd}{\dd t} d_{jki}(\xi,\eta,x,t)\, v_i(\eta,\omega_k(\eta,\xi,\omega_j(\xi))) \dd\xi \dd\eta
	\\ &&\quad
	-\displaystyle \sum_{k\neq j}\sum_{i\neq k}\int_{x_j}^x\int_{\eta}^{x}\d_{\xi}\tilde{d}_{jki}(\xi,\eta,x,t)\,v_i(\eta,\omega_k(\eta,\xi,\omega_j(\xi))) \dd\xi \dd\eta
	\\ &&\quad
	+\displaystyle \sum_{k\neq j}\sum_{i\neq k}\int_{x_j}^x\left [\tilde{d}_{jki}(\xi,\eta,x,t)\, v_i(\eta,\omega_k(\eta,\xi,\omega_j(\xi)))\right ]_{\xi=\eta}^{\xi=x} \dd\eta.
\end{eqnarray*}
According to this representation, we conclude that the function $\d_t[(B^2u^l)_j(x,t)]$ converges to 
$\d_t[(B^2u)_j(x,t)]$ as $l\to\infty$ in $C_{per}(\Pi)$ , which completes the proof for $B^2u$.

For the function $QBu$, we start from the representation obtained from
the formula for $QBu$  after changing the order of integration, namely:
\beq\label{QB}
\begin{array}{rcl}
	[QBu]_{j}(x,t)
	&=&\displaystyle-c_{j}(x_j,x,t)\sum_{k=1}^n\sum_{i\neq  k}
	\int_{0}^{1}\int_{\xi}^{1-x_k}p_{jk}(\eta,\omega_j(x_j))\,d_{k}(\xi,\eta,\omega_j(x_j))   \\ [2mm]
	&&\displaystyle\times \,b_{ki}(\xi,\omega_k(\xi,\eta,\omega_j(x_j)))\,u_i(\xi,\omega_k(\xi,\eta,\omega_j(x_j)))\dd\eta \dd\xi.   
\end{array}
\ee
Again, for $j\le n$ and $v\in C^1_{per}(\Pi;\R^n)$, it holds
\beq\label{direct_QB}
\left(\d_t+a_j(x,t)\d_x\right)\left[QBv\right]_j(x,t)=-b_{jj}(x,t)\left[QBv\right]_j(x,t).
\ee	
Denote by $Q_{jki}\in\mathcal{L}(C_{per}(\Pi))$ the double-integral operators appearing on the right-hand side of~\reff{QB}. Similarly to the above, for any fixed $j\le n$, $k\le n$,  $i\ne k$, and $v\in C^1_{per}(\Pi;\R^n)$, we compute the derivative
\beq\label{Qjki}
\begin{array}{rcl}
\d_t[Q_{jki}v](x,t)
&	=&\displaystyle\int_{0}^{1}\int_{\xi}^{1-x_k}\frac{\dd}{\dd t}\Bigl[p_{jk}(\eta,\omega_j(x_j))\,d_{k}(\xi,\eta,\omega_j(x_j))\, b_{ki}(\xi,\omega_k(\xi,\eta,\omega_j(x_j))) \Bigr] \\ [5mm]
&&\qquad\qquad\displaystyle\times \,v_i(\xi,\omega_k(\xi,\eta,\omega_j(x_j)))\dd\eta \dd\xi \\ [4mm]
&&	\displaystyle +\int_{0}^{1}\int_{\xi}^{1-x_k}p_{jk}(\eta,\omega_j(x_j))\,d_{k}(\xi,\eta,\omega_j(x_j))\, b_{ki}(\xi,\omega_k(\xi,\eta,\omega_j(x_j)))  \\ [5mm]
&&	\qquad\qquad\displaystyle\times 
\,\d_2v_i(\xi,\omega_k(\xi,\eta,\omega_j(x_j)))\,\d_3\omega_k(\xi,\eta,\omega_j(x_j))\,\d_t\omega_j(x_j)\dd\eta \dd\xi
\end{array}
\ee
and, similarly to the above, we derive a suitable representation formula for $\d_2v_i$. Specifically,
$$
\frac{\dd}{\dd\eta}v_i(\xi,\omega_k(\xi,\eta,\omega_j(x_j)))=\d_2v_i(\xi,\omega_k(\xi,\eta,\omega_j(x_j)))
\,\d_2\omega_k(\xi,\eta,\omega_j(x_j))
$$
Using  the identity \reff{om23}, we obtain
$$
\d_2v_i(\xi,\omega_k(\xi,\eta,\omega_j(x_j)))\,\d_3\omega_k(\xi,\eta,\omega_j(x_j))=-a_j(\eta,\omega_j(x_j))\,\frac{\dd}{\dd\eta}v_i(\xi,\omega_k(\xi,\eta,\omega_j(x_j))).
$$
Then the second integral term on the right-hand side of \reff{Qjki} takes the form
\beq\label{Q2}
\begin{array}{cc}
\displaystyle-\d_t\omega_j(x_j)\int_{0}^{1}\int_{\xi}^{1-x_k}a_j(\eta,\omega_j(x_j))\,p_{jk}(\eta,\omega_j(x_j))\,d_{k}(\xi,\eta,\omega_j(x_j))\, b_{ki}(\xi,\omega_k(\xi,\eta,\omega_j(x_j)))  \\ [5mm]
	\displaystyle\times 
\,\frac{\dd}{\dd\eta}v_i(\xi,\omega_k(\xi,\eta,\omega_j(x_j)))\dd\eta \dd\xi
\\ [5mm]
\displaystyle=- \d_t\omega_j(x_j) \int_{0}^{1}\Bigl[a_j(\eta,\omega_j(x_j))\,p_{jk}(\eta,\omega_j(x_j))\,d_{k}(\xi,\eta,\omega_j(x_j))\, \left[b_{ki}v_i\right](\xi,\omega_k(\xi,\eta,\omega_j(x_j)))\Bigr]_{\eta=\xi}^{\eta=1-x_k}\dd\xi\\ [5mm]
\displaystyle +\d_t\omega_j(x_j)\int_{0}^{1}\int_{\xi}^{1-x_k}
\frac{\dd}{\dd\eta}\Bigl[a_j(\eta,\omega_j(x_j))\,p_{jk}(\eta,\omega_j(x_j))\,d_{k}(\xi,\eta,\omega_j(x_j))\, b_{ki}(\xi,\omega_k(\xi,\eta,\omega_j(x_j)))
\Bigr]  \\ [5mm]
\displaystyle\times 
\,v_i(\xi,\omega_k(\xi,\eta,\omega_j(x_j)))\dd\eta \dd\xi.
\end{array}
\ee
We conclude that $\d_t[Q_{jki}u^l]$ converges to $\d_t[Q_{jki}u]$ in $C_{per}(\Pi)$ for all  $j,k,i\le n$ with $i\ne k$.  Combining \reff{QB} with \reff{direct_QB}, \reff{Qjki}, and \reff{Q2}, together with the regularity assumptions on the data of
the original problem, completes the proof of the required regularity of
$QBu$.

	We proceed similarly with the function $BQu$.  By the definitions of $B$ and $Q$, the composition $BQ$  is given by the formula
	\begin{equation}\label{dro}
		\begin{array}{rcl}
			[BQu]_{j}(x,t)
			&=&-\displaystyle \sum_{k\neq j}\sum_{i=1}^{n}\int_{0}^{1}\int_{x_j}^{x}d_{j}(\xi,x,t)\,b_{jk}(\xi,\omega_j(\xi))\,c_{k}(x_k,\xi,\omega_j(\xi))
		  \\
		&&	\displaystyle \times \,p_{ki}(\eta,\omega_k(x_k,\xi,\omega_j(\xi)))\,u_{i}(\eta,\omega_k(x_k,\xi,\omega_j(\xi)))\dd\xi \dd\eta, \;\;\; j\le n.    
		\end{array}
	\end{equation}
	The integral operators in (\ref{dro}) are  similar  to those contributed into $B^2u$ and, therefore,
	the proof follows along the same line as the proof for $B^2u$.

In the next step, we treat the function $B\CC u$. 	If $v\in C^1_{per}(\Pi)$, then
\beq\label{directional2}
\left(\d_t+a_j(x,t)\d_x\right)	[B\CC v]_{j}(x,t)=-b_{jj}(x,t)	[B\CC v]_{j}(x,t).
\ee	
Hence, it is sufficient to prove that $\d_t[B\CC u^l]_{j}$ converges to $\d_t[B\CC u]_{j}$ in $C_{per}(\Pi)$ for any  $j\le n$
and any sequence $(u^l)\subset C^1_{per}(\Pi;\R^n)$
converging to $u$ in $C_{per}(\Pi;\R^n)$. 

Note that, for any $j\le n$, the $j$-th component of the operator  $B\CC$ can be represented as a sum over $k\ne j$ of operators $R_{jk}\in\LL(C_{per}(\Pi))$, 
 given by  
	\beq \label{f315}
[R_{jk}v](x,t)=	\int_{x}^{x_j} e_{jk}(\xi,x,t)\,[Rv]_k(\omega_k(x_k,\xi,\omega_j(\xi))) \dd\xi,
	\ee
	where 
	$$
	e_{jk}(\xi,x,t)=d_j(\xi,x,t)\,b_{jk}(\xi,\omega_j(\xi))\,c_k(x_k,\xi,\omega_j(\xi)).
	$$
Hence, the proof  follows from  the convergence  of $\d_t[R_{jk}u^l]$ to $\d_t[R_{jk}u]$
for all $j,k\le n$ with $k\ne j$.
	To prove this convergence, fix $j\le n$, $k\ne j$, and $v\in C^1_{per}(\Pi)$, and compute the derivative
	\begin{equation}\label{BC0}	\begin{array}{ccc}
		&	\displaystyle \d_t[R_{jk}v](x,t)=\int_{x}^{x_j} \d_te_{jk}(\xi,x,t)\,[Rv]_k(\omega_k(x_k,\xi,\omega_j(\xi)))\dd\xi&
		\\ [4mm]&
		+\displaystyle \int_{x}^{x_j} e_{jk}(\xi,x,t)\,
		\d_3\omega_k(x_k,\xi,\omega_j(\xi))\,\d_t\omega_j(\xi)
	\displaystyle	\left[R^\prime v\right]_k(\omega_k(x_k,\xi,\omega_j(\xi))) \dd\xi&	\\ [4mm]&
	+\displaystyle \int_{x}^{x_j} e_{jk}(\xi,x,t)
	\d_3\omega_k(x_k,\xi,\omega_j(\xi))\d_t\omega_j(\xi)
	\displaystyle	 \left[R \,\d_2v\right]_k(\omega_k(x_k,\xi,\omega_j(\xi))) \dd\xi,&
\end{array}
	\end{equation}
where
\begin{eqnarray*}
	\displaystyle \left[R^\prime v\right]_k(t)=\sum\limits_{i=m+1}^nr_{ki}^\prime(t) v_i(0,t)+\sum\limits_{i=1}^mr_{ki}^\prime(t) v_i(1,t).
\end{eqnarray*}
We now transform the third integral in \reff{BC0} using the following
representation of the summands labeled by $i\le n$ that contribute to $\left[R \,\d_2v\right]_k(\omega_k(x_k,\xi,\omega_j(\xi)))$:
		\begin{eqnarray*}
		\lefteqn{
	r_{ki}(\omega_k(x_k,\xi,\omega_j(\xi)))\,\d_2v_i(1-x_i,\omega_k(x_k,\xi,\omega_j(\xi))}
		\\ &&
		=\frac{[a_ja_k](\xi,\omega_j(\xi))}{[a_j-a_k](\xi,\omega_j(\xi))\, \d_3\omega_k(x_k,\xi,\omega_j(\xi))}
	r_{ki}(\omega_k(x_k,\xi,\omega_j(\xi)))\frac{\dd}{\dd\xi}\, v_i(1-x_i,\omega_k(x_k,\xi,\omega_j(\xi)),
	\end{eqnarray*}
the formula being true due to \reff{char}, \reff{dx}, and \reff{dt}.
		Set
	$$
	\tilde{e}_{jk}(\xi,x,t)
	=\frac{e_{jk}(\xi,x,t)\,	\d_t\omega_j(\xi)\,[a_ja_k](\xi,\omega_j(\xi))}{[a_j-a_k](\xi,\omega_j(\xi)) 
		}
	$$
	and rewrite the third summand in \reff{BC0}, after integration by parts, as follows:
	\begin{eqnarray}
	\lefteqn{	\displaystyle\Bigl[\tilde e_{jk}(\xi,x,t) \left[Rv\right]_k(\omega_k(x_k,\xi,\omega_j(\xi)))\Bigr]_{\xi=x}^{\xi=x_j}}
		\nonumber\\  	&&
		-\displaystyle\sum_{i=1}^{n}\int_{x}^{x_j}\frac{\dd}{\dd\xi}\Bigl(\tilde{e}_{jk}(\xi,x,t)\,
			r_{ki}(\omega_k(x_k,\xi,\omega_j(\xi)))
		\Bigr)v_i(1-x_i,\omega_k(x_k,\xi,\omega_j(\xi))\dd\xi.
		\label{BC1}
	\end{eqnarray}
	The desired convergence  of $\d_t[R_{jk}u^l]$ to $\d_t[R_{jk}u]$ now easily follows from \reff{BC0} and \reff{BC1}. 
	
	The claimed regularity of  $B\CC u$ follows by an analogous argument.
		\end{subproof}
	
	\begin{claim} \label{C1}
The continuous solution $u$ to the problem  (\ref{lin})--\reff{per} belongs to $C^1_{per}(\Pi;\R^n)$.
\end{claim}
	\begin{subproof} Note that Condition 2 of Lemma \ref{lem:smoothing}
		follows immediately from the definition~\reff{CDF} of the operators
		$\CC$, $Q$ and $B$. Hence, in view of
		  Claims \ref{I-C} and \ref{B2}, the operator equation \reff{oper}, and
		 Lemma \ref{lem:smoothing}, it remains to show that 
$\F(f,u)\in C^1_{per}(\Pi;\R^n)$. Due to~\reff{CDF}, it suffices to show that the function
\beq\label{m}
\begin{array}{ll}
\displaystyle\int_{x_j}^{x}d_j(\xi,x,t) \left[Mu\right]_j(\xi,\omega_j(\xi,x,t))\dd\xi\\ [3mm]
\displaystyle\qquad=		\sum_{k=1}^{n}\int_{x_j}^{x}  d_j(\xi,x,t)\int_{0}^{\xi}m_{jk}(\eta,\xi,\omega_j(\xi,x,t))\,u_k(\eta,\omega_j(\xi,x,t))\dd\eta\dd\xi
\end{array}
	\ee
		belongs to $C^1_{per}(\Pi;\R^n)$ for every $j\le n$. 
The conclusion then follows, as above, from the fact that the double integrals in \reff{m} are taken over two transverse curves. In
particular, transversality yields the following nondegenerate change of
variables for every $j\le n$:
\begin{equation}\label{change1}
	\xi\in[0,1]\mapsto\tau=\omega_j(\xi,x,t).
\end{equation}
On the account of  \reff{char}, this implies that
$
\dd\xi=a_j(\widetilde\om_j(\tau),\tau)\dd\tau,
$
where $\xi=\widetilde\om_j(\tau)=\widetilde\om_j(\tau,x,t)$ denotes the inverse to \reff{change1}.  Hence,
$
\om_j\left(\widetilde\om_j(\tau,x,t),x,t)\right)=\tau.
$
Let $j\le m$ be fixed (the case $j> m$ can be treated similarly).
After interchanging the order of integration and performing the change of variables \reff{change1}, the right-hand side of \reff{m} takes the form
\beq\label{mm}
\sum_{k=1}^{n}\int_{0}^{x}\int_{\om_j(\eta)}^{t}a_j(\widetilde\om_j(\tau),\tau)\,
d_j(\widetilde\om_j(\tau),x,t)\,m_{jk}(\eta,\widetilde\om_j(\tau),\tau)\,u_k(\eta,\tau)\dd\tau\dd\eta.
\ee
This representation and the regularity conditions  on the data imply that the function
\reff{m} belongs to  $C^1_{per}(\Pi)$, as desired.

Therefore, applying Lemma \ref{lem:smoothing} completes the proof.
		\end{subproof}

The next two claims are devoted to the 	$C^2_{per}(\Pi;\R^n)$-regularity of the function $u\in C^1_{per}(\Pi;\R^n)$.
Now, we apply the  abstract Lemma \ref{lem:smoothing} with $m=1$, $k=2$, $V=C^1_{per}(\Pi;\R^n)$, and $U=C^2_{per}(\Pi;\R^n)$.
	\begin{claim} \label{I-C_1}
		 If for $v\in C^1_{per}(\Pi;\R^n)$ it holds $(I-\CC)v\in C^2_{per}(\Pi;\R^n)$, then  $v\in C^2_{per}(\Pi;\R^n)$.
	\end{claim}
	
	\begin{subproof} 
				Let $v\in C^1_{per}(\Pi;\R^n)$ 
				be such that $(I-\CC)v\in C^2_{per}(\Pi;\R^n)$. Then 	there exists $g\in C^2_{per}(\Pi;\R^n)$
				such that  the  equalities \reff{simpl} hold. Hence, the functions $z$ and $\tilde g$, defined by \reff{z},
				obey the regularities $z\in C^1_{per}(\R;\R^n)$ and $\tilde g\in C_{per}^2(\R;\R^n)$.
				Accordingly to \reff{simpl},  the proof reduces to showing that $z\in C^2_{per}(\R;\R^n)$.  
			By Claim \ref{I-C} and its proof,	 $z$
			satisfies  \reff{zD} or, the same,
			\beq\label{I-G2}
	z^{\prime}=	G_1z^{\prime}+	W_1z+\tilde g^\prime.
			\ee
Therefore, it suffices to show that $z^{\prime}$, satisfying \reff{I-G2}, belongs to $C^1_{per}(\R;\R^n)$. 

			Note that $	W_1z+\tilde g^\prime\in C^1_{per}(\Pi;\R^n)$.
		Now, if $y\in C^1_{per}(\Pi;\R^n)$,  then, using \reff{Ci}, we can compute the following derivatives pointwise:
				\begin{equation}\label{dtG1}
				\displaystyle	\frac{\dd}{\dd t} [G_1y]_j(t) = [G_2y^{\prime}]_j(t)
				+	  [W_2 y]_j(t),
					\quad  j\le n,
				\end{equation}
				where 
				  the operator $W_2 \in \mathcal L(C_{per}(\R; \R^n))$ is defined by
				\beq\label{W2}
				[W_2 w]_j(t) =  c_j^1(x_j,1 - x_j,t) \sum\limits_{k=1}^n\frac{\dd}{\dd t}\Bigl[c_j^1(x_j,1 - x_j,t)\, r_{jk}(\omega_j(x_j, 1 - x_j, t))\Bigr]w_k(\omega_j(x_j, 1 - x_j, t))
				\ee
and the functions
$c_j^1$ are given by \reff{cd}.	

The remainder of the proof follows the same argument as the proof of  Claim~\ref{I-C}, but now in the setting \reff{I-G2}--\reff{W2} instead of
\reff{simpl1}--\reff{W}, with $z^\prime$, $W_1z+\tilde g^\prime$, $G_1$, $G_2$, and $W_2$ playing the roles of 
 $z$, $\tilde g_j$, $G_0$, $G_1$, and $W_1$, respectively.

This completes the proof.
	\end{subproof}

\begin{claim} \label{B22}
The functions $Q\CC u, B\CC u, Q^2u,B^2u,QBu$, and $BQu$ belong to $C^2_{per}(\Pi;\R^n)$. 
\end{claim}
	\begin{subproof}
By Claim \ref{C1}, the continuous solution  $u$ belongs to $C^1_{per}(\Pi;\R^n)$. The claim now  follows by the same argument as in the proof of Claim  \ref{B2}. The only additional ingredient is that, besides the $C^1_{per}(\Pi;\R^n)$-regularity of $u$,
we use  the formulas for the $t$-derivatives and the directional derivatives along characteristic curves obtained for the functions $Q^2u$, $Q\CC u$, $B\CC u$, $B^2u$, $QBu$, and $BQu$.

For example, to prove  the $C^2$-regularity for $Q^2u$ (the remaining functions being treated analogously), we  apply the formula \reff{PP}, which is now valid with $u$ in place of $u^l$. Hence, the functions  $\d_t[P_{jk}P_{ki}u]$ belong to $C^1_{per}(\Pi)$. It then follows
from \reff{directional} and its $x$-derivative that $P_{jk}P_{ki}u\in C^2_{per}(\Pi)$. Consequently, $Q^2u\in C^2_{per}(\Pi;\R^n)$, as desired. 
	\end{subproof}

	\begin{claim} \label{7} The classical  $C^1_{per}(\Pi;\R^n)$-solution $u$ to the problem  (\ref{lin})--\reff{per} belongs to $C^2_{per}(\Pi;\R^n)$.
	\end{claim}
	
	\begin{subproof}
	First note that the function $\F(f,u)$ belongs to $C^2_{per}(\Pi;\R^n)$. This follows directly from the definition \reff{CDF} of $\F$,  the fact that $u\in C^1_{per}(\Pi;\R^n)$ and  the representation formula \reff{mm} for \reff{m}.
	 Combining this observation with
	Claims \ref{I-C_1} and \ref{B22},  and applying Lemma~\ref{lem:smoothing}
	  with $m=1$, $k=2$, $U=C_{per}^2(\Pi;\R^n)$, and $V=C^1_{per}(\Pi;\R^n)$,
	  we conclude that
	$u\in C^2_{per}(\Pi;\R^n)$. 
	\end{subproof}
	Proceeding in the same way as in the proof of the $C^2$-regularity of $u$, one obtains $u\in C^l_{per}(\Pi;\R^n)$ for any $l\ge 3$.
	This completes the proof of  Theorem \ref{thm:lin}.

\begin{rem}\label{reg}
	(On the optimality of the regularity assumptions in
	Theorem~\ref{thm:lin}).
	\rm
	For simplicity of exposition, Theorem~\ref{thm:lin} (and likewise
	Theorems~\ref{thm:semilin} and \ref{thm:quasilin}) is formulated under
	sufficient, though not optimal, regularity assumptions on the data.
	The optimal regularity requirements can be read off from the proof of
	Theorem~\ref{thm:lin} (and similarly for Theorems~\ref{thm:semilin} and
	\ref{thm:quasilin}). In particular, the $C^1$- and $C^2$-regularity
	statements of Theorem~\ref{thm:lin} can be refined as follows, and the
	same applies to higher regularity results.

	\begin{thm}\label{regC1}
	{\bf 1.}	
			Assume that,  for all $j,k\le n$, the coefficients
			$a_j$,   $b_{jk}$,  $p_{jk}$, $r_{jk}$, and $h_j$
			belong to  $C^{1}_{per}$, the coefficients $f_j$ belong to $C_{per}\cap C_t^1$,
			the coefficients	$m_{jk}$   belong to  $C_{per}\cap C_x^1$, all   in their respective domains.
			If the conditions \reff{aj} and \reff{nonrez}
			are satisfied 
			for all $j\le n$,  $x\in[0,1]$, $t\in\R$, and $i=1$,
			then any continuous  solution to the problem~(\ref{lin})--\reff{per} 
			belongs to~$C^1(\Pi;\R^n)$.\\
			{\bf 2.}
			Assume that,  for all $j,k\le n$, the coefficients
			 $r_{jk}$ and $h_j$
			belong to  $C^{2}_{per}$,  the coefficients $a_j$, $b_{jk}$ belong to $C^{1}_{per}$ with  $\d_ta_j, \d_tb_{jk}\in C^{1}_{per}$,  the functions $f_j$ belong to $C^1_{per}\cap C_t^2$,
			the coefficients	$m_{jk}$   belong to  $C_{per}\cap C_x^2$ and have the continuous derivatives $\d_tm_{jk}$, 
			and the coefficients $p_{jk}$ belong to $C_{per}\cap C_t^2$ with $\d_xp_{jk}\in C_t^{1}$, all   in their respective domains.
			If the conditions \reff{aj} and \reff{nonrez}
			are satisfied 
			for all $j\le n$,  $x\in[0,1]$, $t\in\R$, $i=1$, and $i=2$,
			then any continuous  solution to the problem (\ref{lin})--\reff{per} 
			belongs to~$C^2(\Pi;\R^n)$.
	\end{thm}
\end{rem}

\subsection{Equivalence of the continuous and the distributional solution concepts}

In the next Subsection \ref{2} we prove the semilinear Theorem \ref{thm:semilin}.
To this end,  we here  establish an equivalence result
between the continuous and distributional formulations of the linear problem
\reff{lin}--\reff{per} (Lemma~\ref{equiv}), using an argument that extends to a broader class of hyperbolic problems.
The proof is based on the following auxiliary lemma.

\begin{lem}\label{const}
	For every $j\le n$, let $	a_j \in C^2(\Pi) $
	satisfy
	$$
	0<c_0 \le |a_j(x,t)| \le c_1
	\qquad \mbox{for all }\  (x,t)\in \Pi
	$$
	where $c_0,c_1>0$ are constants.
	Let $U_j\in C(\Pi)$ satisfy
	\beq\label{transport}
	\partial_t U_j + a_j(x,t)\partial_x U_j =0
	\qquad \mbox{in } \ \mathcal D'\left(\Pi^\circ\right),
	\ee
where  $\Pi^\circ$	denotes the interior of $\Pi$. Then
	\beq\label{traces}
	U_j(x,t)
	=
	U_j\bigl(x_j,\omega_j(x_j,x,t)\bigr)
	\qquad \mbox{for all }\  (x,t)\in\Pi.
	\ee
\end{lem}
\begin{proof}
	Fix $j\le n$ and $(x_0,t_0)\in\Pi$ and write $\omega_j^0(x)=\omega_j(x,x_0,t_0)$.
	Let $\eta \in C_c^\infty((0,1))$ and $\rho\in C_c^\infty((-1,1))$, and
	let $\rho_\varepsilon=\frac{1}{\eps}\rho\left(\frac{x}{\eps}\right)$ be a standard delta-mollifier.
	Define
	\beq\label{theta}
	\Theta_j(x,t)
	=
	\exp\!
	\int^x_{x_j}
	\left(\frac{\partial_t a_j}{a_j^2}\right)(\xi,\omega_j(\xi,x,t))\dd \xi
	\ee
	and
	\beq\label{phi_eps}
	\phi_\varepsilon(x,t)
	=\eta(x)\,\frac{\Theta_j(x,t)}{a_j(x,t)}
	\,\rho_\varepsilon\bigl(t-\omega_j^0(x)\bigr).
	\ee
	Note that $\phi_\varepsilon$ is
	supported in the $\eps$-neighborhood of the characteristic curve $t=\omega_j^0(x)$.	
	Since  $	a_j \in C^2(\Pi) $,
	it follows that $(\phi_\varepsilon)\subset C^1_c(\Pi)$. Moreover, since 
	$U_j$ is continuous, then the equality \reff{transport} implies that
	\beq\label{distr_eps}
	\int_\Pi U_j\bigl(\partial_t\phi_\varepsilon + \partial_x(a_j\phi_\varepsilon)\bigr)\dd x\dd t=0
	\ee
	for all $\eps>0$. 
	
	By direct calculations, from the formula \reff{theta} we obtain that
	$$
	\frac{1}{a_j}
	\partial_t\Theta_j+\d_x\Theta_j-
	\frac{\partial_t a_j}{a_j^2}\,
	\Theta_j=0
	$$
	and, hence, from the formula \reff{phi_eps} combined with the mean value theorem, we derive  that
	\begin{eqnarray*}
		\partial_t\phi_\varepsilon + \partial_x(a_j\phi_\varepsilon)
		&=&\eta'(x)\Theta_j(x,t)\rho_\varepsilon(t-\omega_j^0(x))\\
		&&	+\eta(x)\frac{\Theta_j(x,t)}{a_j(x,t)a_j(x,\omega_j^{0}(x))}
		\,\rho'_\varepsilon(t-\omega_j^0(x))
		\left[a_j(x,\omega_j^{0}(x)) - a_j(x,t)\right] \\
		&=&\eta'(x)\Theta_j(x,t)\rho_\varepsilon(t-\omega_j^0(x))
		\\ &&-\eta(x)\frac{\Theta_j(x,t)\,(t-\omega_j^0(x))}{a_j(x,t)a_j(x,\omega_j^{0}(x))}
		\rho'_\varepsilon(t-\omega_j^0(x))
		\int_{0}^{1}\d_ta_j(x,\omega_j^0(x)+\sigma(t-\omega_j^0(x)))\dd\sigma,
	\end{eqnarray*}
	where
	$$
	\rho'_\varepsilon(t-\omega_j^0(x))=\frac{1}{\eps^2}\rho'\left(\frac{t-\omega_j^0(x)}{\eps}\right).
	$$

	Set
	$$
	F_j(x,t)=\eta(x)\frac{\Theta_j(x,t)}{a_j(x,t)\,a_j(x,\omega_j^{0}(x))}
	\int_{0}^{1}\d_ta_j(x,\omega_j^0(x)+\sigma(t-\omega_j^0(x)))\dd\sigma
	$$
	and rewrite the  integral in \reff{distr_eps} as follows:
	$$\begin{array}{rcl}
		\displaystyle\int_\Pi U_j\bigl(\partial_t\phi_\varepsilon + \partial_x(a_j\phi_\varepsilon)\bigr)\dd x\dd t&=&\displaystyle
		\int_\Pi U_j(x,t)\eta'(x)\Theta_j(x,t)\,\rho_\varepsilon(t-\omega_j^0(x))\dd x\dd t\\ &&\displaystyle
		-\int_\Pi U_j(x,t)\,F_j(x,t)\,\rho_\varepsilon(t-\omega_j^0(x))\dd x\dd t.
	\end{array}$$
	The last integral, after changing the variable $t$ to 
	$s=t-\omega_j(x)$, takes  the form
	$$\begin{array}{ll}
		\displaystyle\int_\Pi U_j(x,\omega_j^0(x)+\eps s)\,F_j(x,\omega_j^0(x)+\eps s)\,s\,\rho'(s)\dd x\dd s\\ [3mm]=\displaystyle
		\int_\Pi U_j(x,\omega_j^0(x)+\eps s)\,F_j(x,\omega_j^0(x)+\eps s)\,(s\,\rho(s))^\prime\dd x\dd s\\ [3mm]
		-\displaystyle
		\int_\Pi U_j(x,\omega_j^0(x)+\eps s)\,F_j(x,\omega_j^0(x)+\eps s)\,\rho(s)\dd x\dd s.
	\end{array}
	$$
	Passing to the limit $\varepsilon\to 0$, the first summand on the right-hand side vanishes, and we get 
	$$
	\int_\Pi U_j(x,\omega_j^0(x))\,F_j(x,\omega_j^0(x))\,s\,\rho'(s)\dd x\dd s=
	-
	\int_0^1 U_j(x,\omega_j^0(x))\,F_j(x,\omega_j^0(x))\dd x.
	$$
	Now, on the account of the formula \reff{distr_eps}, we obtain
	$$
	\begin{array}{rcl}
		\displaystyle	0&=&	\displaystyle\lim_{\eps\to 0}\int_\Pi U_j\bigl(\partial_t\phi_\varepsilon + \partial_x(a_j\phi_\varepsilon)\bigr)\dd x\dd t
		\displaystyle=\int_0^1 U_j(x,\omega_j^0(x))
		\eta'(x)\Theta_j(x,\omega_j^0(x))\dd x\\\displaystyle
		&&	+	\displaystyle
		\int_0^1 U_j(x,\omega_j^0(x))\left(\frac{\d_ta_j}{a_j^2}\right)(x,\omega_j^0(x))
		\eta(x)\,\Theta_j(x,\omega_j^0(x))\dd x\\
		\displaystyle&=&	\displaystyle
		\int_0^1U_j(x,\omega_j^0(x))\frac{\dd}{\dd x}\left(\eta(x)\,\Theta_j(x,\omega_j^0(x))\right)\dd x.
	\end{array}
	$$
	Finally, since $\Theta(x,\omega_j^0(x))\ne 0$ and $U_j$ is continuous, then denoting
	$
	v(x)=U_j(x,\omega_j^0(x)),
	$
	we conclude that
	$$
	\int_0^1v(x)\psi'(x)\dd x=0
	$$
	for any $\psi\in C_c^1((0,1))$. This means that $v$ is constant, which implies 
	$v(x)=v(x_j)$. Therefore,
	$$
	U_j(x,\om_j^0(x))=U_j(x_j,\om_j^0(x_j))=U_j(x_j,\om_j(x_j,x_0,t_0)).
	$$
	Since $(x_0,t_0)\in\Pi$ is arbitrary, the desired equality \reff{traces} now easily follows.
\end{proof}

\begin{rem} \rm 
	The proof of Lemma~\ref{const} is based on direct calculations and describes
	the reduction of distributional solutions of \reff{transport}, viewed as
	functions of two variables, to the corresponding one-variable functions along
	characteristic curves. Alternatively, the lemma could be derived from standard
	results on first-order linear partial differential equations
	\cite[Section~6.1]{Horm}, which imply that distributional solutions of
	\reff{transport} are constant along the integral curves of the {\it smooth} vector
	field $\partial_t+a_j(x,t)\partial_x$.
\end{rem}

\begin{lem}\label{equiv}
	A continuous function $u\in C_{\mathrm{per}}(\Pi;\mathbb R^n)$ is
	a continuous solution
	to the problem  \reff{lin}--\reff{per} (see Definition \ref{cont}) if and only if
	$u$ satisfies the differential system \reff{lin} in a distributional sense and the conditions \reff{bc_lin} and \reff{per} pointwise.
\end{lem}

\begin{proof}
	{\it Necessity.} Let $u$ be a continuous solution
	to the problem  \reff{lin}--\reff{per}. By Definition~\ref{cont},
	$u$ satisfies the equation \reff{oper} pointwise. Using
	the notations introduced in \reff{cd} and \reff{CDF}, it follows directly that the conditions \reff{bc_lin} and \reff{per} are satisfied pointwise as well.
	It remains to verify that $u$ satisfies (\ref{lin}) in the distributional sense. To this end, 
	let $(u^l)\subset C_{per}^{1}\left(\Pi;\R^n\right)$ be   any sequence converging to
	$u$ in $C_{per}\left(\Pi;\R^n\right)$.  Then, for any   $j\le n$ and any compactly supported function
	$\phi\in C_c^\infty((0,1)\times\R)$,
	it holds
	\begin{eqnarray*}
		& & \left\langle (\d_t+a_j\d_x)u_j,\phi\right\rangle = \left\langle u_j,-\d_t\phi-\d_x(a_j\phi)\right\rangle
		= \lim_{l\to\infty}\left\langle u_j^l,-\d_t\phi-\d_x(a_j\phi)\right\rangle\\
		&&\qquad =
		\lim_{l\to\infty}\left\langle 
		[\CC u^l]_j+[Qu^l]_j+[Bu^l]_j+[\H h]_j+[\F(f,u^l)]_j,-\d_t\phi-\d_x(a_j\phi)\right\rangle	\\
		&&\qquad =
		\lim_{l\to\infty}\left\langle -b_{jj}\left[
		[\CC u^l]_j+[Qu^l]_j+[Bu^l]_j+[\H h]_j+[\F(f,u^l)]_j\right],\phi\right\rangle\\
		&& \qquad\ \ \ +
		\lim_{l\to\infty}\left\langle 
		-\sum_{k\not=j}b_{jk}u_k^l+f_j-Mu^l,\phi\right\rangle
		=\left\langle 
		-\sum_{k=1}^nb_{jk}u_k+f_j-Mu,\phi\right\rangle,
	\end{eqnarray*}
	as desired.
	Here we used \reff{oper} twice and  the identity
	\beq\label{2k}
	(\d_t+a_j(x,t)\d_x)\psi(\om_j(\xi,x,t))=0,
	\ee
	which holds  for all $j\le n$, $\xi,x\in[0,1]$, $t\in\R$, and  $\psi\in C^1(\R)$.

	{\it Sufficiency.} Assume that a continuous function $u$
	satisfies the equation \reff{lin} in a distributional sense and the conditions \reff{bc_lin} and \reff{per} pointwise. 
	Fix any and $j\le n$. We first show that the function 
	\begin{equation}\label{11k}
		U_j(x,t)=c_j^{-1}(x_j,x,t)	\Bigl[u_j(x,t)-[Bu]_j(x,t)-\left[\F(f,u)\right]_j(x,t)\Bigr]
	\end{equation}	
	is a constant along the characteristic curve $\om_j(\xi,x,t)$ passing through $(x,t)$. To this end, we prove that the distributional directional derivative of $U_j$ in the direction of the vector field $(1,a_j(x,t))$ vanishes. The key point is that, since $u$ is continuous and satisfies \reff{lin} in the distributional sense, the directional derivative  $(\partial_t  + a_j(x,t)\partial_x)u_j$ is represented by the continuous function
	$$
	f_j-\sum_{k=1}^nb_{jk}u_k-[Mu]_j.
	$$
	Consequently, the directional derivative $(\partial_t  + a_j(x,t)\partial_x)u_j$ is well-defined in the sense of distributions. Let
	$\phi\in C_c^\infty\left((0,1)\times\R\right)$ be arbitrary
	and let $(u^l)\subset C_{per}^{1}\left(\Pi;\R^n\right)$ be   any sequence  converging to
	$u$ in $C_{per}\left(\Pi;\R^n\right)$. 
	Then, testing $(\d_t+a_j\d_x)U_j$ against $\phi$ and using \reff{11k} and \reff{oper}, we obtain
	\begin{eqnarray*}
		& & \left\langle (\d_t+a_j\d_x)U_j,\phi\right\rangle=\left\langle U_j,-\d_t\phi-\d_x(a_j\phi)\right\rangle
		\\
		&&\qquad =\lim_{l\to\infty}\left\langle c_j^{-1}(x_j,x,t)	\left(u_j^l-[Bu^l]_j-[\F(f,u^l)]_j\right),-\d_t\phi-\d_x(a_j\phi)\right\rangle\\
		&&\qquad =\lim_{l\to\infty}\left\langle \left[	u_j^l-[Bu^l]_j-[\F(f,u^l)]_j\right](\d_t+a_j\d_x)c_j^{-1}(x_j,x,t),\phi\right\rangle\\ && \qquad\ \ \ \ +
		\lim_{l\to\infty}\left\langle c_j^{-1}(x_j,x,t)(\d_t+a_j\d_x)\left[	u_j^l-[Bu^l]_j-[\F(f,u^l)]_j\right],\phi\right\rangle\\
		&& \qquad=\left\langle \left[	u_j-[Bu]_j-\left[\F(f,u)\right]_j\right]b_{jj}c_j^{-1}(x_j,x,t),\phi\right\rangle\\
		&& \qquad\ \ \ \ +\left\langle b_{jj}c_j^{-1}(x_j,x,t)\left[	[Bu]_j+\left[\F(f,u)\right]_j\right],\phi\right\rangle\\
		&&\qquad\ \ \ \ -\biggl\langle c_j^{-1}(x_j,x,t)\biggl[f_j-\sum_{k\ne j}b_{jk}u_k-[Mu]_j\biggr],\phi\biggr\rangle\\
		&&\qquad\ \ \ \ +\biggl\langle c_j^{-1}(x_j,x,t)\biggl[f_j-\sum_{k=1}^nb_{jk}u_k-[Mu]_j\biggr],\phi\biggr\rangle=0.
	\end{eqnarray*}
	Hence, $U_j$ satisfies the homogeneous transport equation \reff{transport}. Applying now Lemma \ref{const} finishes 
	the proof.
\end{proof}

\subsection{Semilinear hyperbolic systems: proof of  Theorem \ref{thm:semilin}}
\label{2}

Our strategy for proving the semilinear Theorem~\ref{thm:semilin}
is based on reducing the semilinear problem \reff{semilin}, \reff{bc}, \reff{per}
to a linear problem of the form \reff{lin}--\reff{per}, to which Theorem~\ref{thm:lin} can then be applied.

Let $u$ be a classical solution to 
the problem (\ref{semilin}),  (\ref{bc}), \reff{per}, and assume that all conditions of Theorem \ref{thm:semilin} are fulfilled for some $l\ge 2$.
Then  the right-hand side of the system \reff{semilin}, and hence the left-hand side, is continuously differentiable in 
$x$ and $t$. Differentiating the system \reff{semilin} with respect to $t$ in the distributional sense
(with the derivatives $\partial_{xt}^2u$ and $\partial_t^2u$ interpreted separately  in the distributional sense), and differentiating the boundary conditions \reff{bc} pointwise with respect to $t$, we set $v=u_t$. Substituting $\partial_xu_j$ from \reff{semilin}, we obtain the following system for $v$  in the distributional sense:
\beq\label{semilin_t}
\begin{array}{rr}
\displaystyle\partial_tv_j+a_j(x,t)\partial_xv_j-\left[\frac{\d_ta_j}{a_j}+\d_{u_j}f_j\right]v_j-\sum_{j\not=k}\d_{u_k}f_j\,v_k-
\d_{n+3}f_j\left[Mv\right]_j(x,t)\\\displaystyle
=-\frac{\d_ta_j}{a_j}f_j\left(x,t,u,\left[Mu\right]_j\right)+\d_tf_j+\d_{n+3}f_j\left[M^\prime u\right]_j(x,t),\quad  j\le n,
	\end{array}
\ee
 and the following boundary and periodicity conditions pointwise:
\beq\label{bc_t}
\begin{array}{ll}
	v_{j}(0,t)= [Rv]_j(t)+[\widetilde Pv]_j(t)+\tilde h_j(t), \quad 1\le j\le m,\\ [1mm]
	v_{j}(1,t)= [Rv]_j(t)+[\widetilde Pv]_j(t)+\tilde h_j(t), \quad m< j\le n,
\end{array}
\ee
and
\beq\label{per_t}
v_j(x,t)=v_j(x,t+2\pi),\quad  j\le n,
\ee
Here the operator $R$ is  defined by \reff{eq:R}, $r_{jk}$ are given by \reff{r_t}, and
 the following notation is used:
\beq\label{not1}
\begin{array}{rcl}
\left[M^\prime u\right]_j(x,t)&=&\displaystyle\sum_{k=1}^{n}\int_{0}^{x}\d_tm_{jk}(\xi,x,t)\,v_k(\xi)\dd\xi,
\\ [3mm]
[\widetilde Pv]_j(t)
&=&\displaystyle\sum_{k=1}^n\int_0^1\widetilde p_{jk}(x,t)\,v_k(x,t)\dd x,\\ [3mm]
\widetilde p_{jk}(x,t) &=&p_{jk}(x,t) \,\d_{n+2}	h_j\Bigl(t, u_{m+1}(0,t),\dots,u_n(0,t), u_1(1,t),\dots,u_m(1,t),\left[Pu\right]_j(t)\Bigr),\\ [3mm]
\widetilde h_j(t)&=&\d_1	h_j\Bigl(t, u_{m+1}(0,t),\dots,u_n(0,t), u_1(1,t),\dots,u_m(1,t),\left[Pu\right]_j(t)\Bigr)\\ [3mm]
&&+\d_{n+2}	h_j\Bigl(t, u_{m+1}(0,t),\dots,u_n(0,t), u_1(1,t),\dots,u_m(1,t),\left[Pu\right]_j(t)\Bigr)\\ [3mm]
&&\displaystyle\times\sum_{k=1}^n\int_0^1\d_t p_{jk}(x,t)\,u_k(x,t)\dd x.
\end{array}
\ee
Rewrite \reff{semilin_t} in the form
\beq\label{semilin_tt}
\partial_tv_j+a_j(x,t)\partial_xv_j+\sum_{k=1}^{n}b_{jk}(x,t)v_k+\d_{n+3}f_j\left(x,t,u,\left[Mu\right]_j\right)\left[Mv\right]_j(x,t)=\widetilde f_j(x,t),\quad j\le n, 
\ee
where
\beq\label{not2}
\begin{array}{rcl}
\displaystyle b_{jj}(x,t)&=&\displaystyle-\left[\frac{A_{jt}}{A_j}+F_{jj}\right],\quad j\le n,\\ [5mm]
b_{jk}(x,t)&=&- F_{jk},\quad j\ne k,\\ [3mm]
\displaystyle\widetilde f_j(x,t)&=&\displaystyle F_{jk}-\frac{A_{jt}}{A_j}F_j
\displaystyle -\d_{n+3}f_j\left[M^\prime u\right]_j(x,t),\quad j\le n,	
\end{array}
\ee
and the functions $A_j$, $A_{jt}$, $F_j$, and $F_{jk}$ 
are defined by \reff{notation_u}.
Since $u$ is fixed, we suppress its dependence in the notation introduced in
\reff{not1} and \reff{not2}.
Note that  the problem \reff{semilin_tt}, \reff{bc_t}, \reff{per_t}, when regarded as a system for  $v$, is linear and is of the form \reff{lin}, \reff{bc}, \reff{per}.
By our construction, the  function $v$ is continuous and is a distributional  solution to
the problem \reff{semilin_tt}, \reff{bc_t}, \reff{per_t}.
Our strategy is, therefore, to show that any continuous function~$v$, which is a distributional solution to \reff{semilin_tt}, \reff{bc_t}, \reff{per_t} (in particular, $v=u_t$) reaches the $C^{l-1}$-regularity,
 since then \reff{semilin} differentiated $l-1$ times in $x$ implies that
$u\in C^l(\Pi,\R^n)$.
 Because of Lemma \ref{equiv}, this will be done if we show that any continuous solution~$v$ 
to  \reff{semilin_tt}, \reff{bc_t}, \reff{per_t} is $C^{l-1}$-regular. 

Let  $v$ be a continuous solution to  \reff{semilin_tt}, \reff{bc_t}, \reff{per_t}. 
The proof proceeds in \(l-1\) steps, each consisting of an application
of the linear Theorem~\ref{thm:lin}.

In the first step, we prove that $v\in C^1(\Pi,\R^n)$ by
 applying Theorem \ref{thm:lin} with $l=1$ to the system \reff{semilin_tt}, \reff{bc_t}, \reff{per_t}. This is possible because, by the regularity assumptions on the
 coefficients of \reff{semilin} and \reff{bc}, together with the fact
 that $u\in C^1(\Pi,\R^n)$,  all coefficients of
 \reff{semilin_tt} and \reff{bc_t} satisfy the regularity assumptions
 required in Theorem~\ref{thm:lin} for $l=1$. Moreover, 
the nonresonance condition \reff{nonrez_semilin} with  $i=2$ is
precisely the nonresonance condition \reff{nonrez} with  $i=1$, where $b_{jj}$ is given by \reff{not2}.
Hence, by Theorem~\ref{thm:lin}, any continuous 
solution $v$ (and, in particular  $v=u_t$) belongs to $C^1(\Pi,\R^n)$. As mentioned above, substituting $v=u_t$ into \reff{semilin} and 
differentiating the resulting system with respect to $x$ shows that $u\in C^2(\Pi,\R^n)$.

In the second step, we prove that $v\in C^2(\Pi,\R^n)$. Now,
since $u\in C^2(\Pi,\R^n)$, the
coefficients of \reff{semilin_tt} and \reff{bc_t} satisfy the regularity assumptions of Theorem \ref{thm:lin} with $l=2$.
 Moreover,  the nonresonance 
condition \reff{nonrez_semilin} with $i=3$  is precisely
 the  condition \reff{nonrez} with $i=2$, where $b_{jj}$ is given by \reff{not2}. Hence, by Theorem \ref{thm:lin}, the continuous 
 solution~$v$ to \reff{semilin_tt}, \reff{bc_t}, \reff{per_t} belongs to $C^2(\Pi,\R^n)$. As in the first step, substituting $v=u_t$ into
 \reff{semilin} and differentiating the resulting system twice with
 respect to $x$, we conclude that $u$ belongs to $C^3(\Pi,\R^n)$.

 Proceeding inductively in this way, after $l-1$ steps we obtain
$v\in C^{l-1}(\Pi,\R^n)$, and hence $u\in C^l(\Pi,\R^n)$, as desired.  
Indeed, in the last step, we assume that  $v\in C^{l-2}(\Pi,\R^n)$ and $u\in C^{l-1}(\Pi,\R^n)$ and prove that $v\in C^{l-1}(\Pi,\R^n)$. The regularity of $u$ implies that
all coefficients of \reff{semilin_tt} and \reff{bc_t} satisfy the regularity assumptions of Theorem~\ref{thm:lin} with $l-1$ in place of $l$. Moreover, as noted above, the nonresonance
conditions for the problem \reff{semilin_tt} are precisely those given
by \reff{nonrez_semilin}. Hence, by Theorem~\ref{thm:lin}, every
continuous solution $v$ to \reff{semilin_tt}, \reff{bc_t}, \reff{per_t}, and, in particular, $v=u_t$ belongs to $C^{l-1}$. 
Finally, substituting $v=u_t$ into
\reff{semilin} and differentiating the resulting system
 $l-1$ times with respect to $x$, we conclude that 
$u\in C^{l}(\Pi,\R^n)$, which completes the proof.

\subsection{Quasilinear hyperbolic systems: proof of  Theorem \ref{thm:quasilin}}\label{3} 

Similarly to the semilinear case, the idea here is to reduce the quasilinear problem to a semilinear one and then  apply the semilinear Theorem \ref{thm:semilin}. 

Let $u$ be a classical $C^2$-solution to 
the problem (\ref{quasilin}),  (\ref{bc}), \reff{per}, and assume that all conditions of Theorem \ref{thm:quasilin} are satisfied for some $l\ge 3$. As in the semilinear case, we differentiate (\ref{quasilin}),  (\ref{bc}), and \reff{per} with respect to $t$, but now all  differentiations are performed  pointwise.
Introducing the notation $v=u_t$, the differentiated system (\ref{quasilin}) takes the form
$$
\begin{array}{rr}
	\displaystyle\partial_tv_j+a_j(x,t,u)\partial_xv_j+\biggl[\d_2a_j+
\sum_{k=1}^n\d_{u_k}a_j\,v_k\biggr]\d_xu_j-\d_{n+3}f_j\left[Mv\right]_j(x,t)\\\displaystyle%
	=\d_tf_j\left(x,t,u,\left[Mu\right]_j\right)+\sum_{k=1}^n\d_{u_k}f_j\,v_k+\d_{n+3}f_j\left[M^\prime u\right]_j(x,t),\quad   j\le n.
\end{array}
$$
After substituting $\d_xu_j$ from \reff{quasilin} and using the notation \reff{notation_u}, the last system reads as follows:
\beq\label{quasilin_t}
\begin{array}{ll}
	\displaystyle\partial_tv_j+A_j(x,t)\partial_xv_j
	=\widetilde f_j\left(x,t,v,\left[Mv\right]_j\right),\qquad j\le n,
\end{array}
\ee
where
\beq\label{not3}
\begin{array}{ll}\
	\displaystyle\widetilde f_j\left(x,t,v,\left[Mv\right]_j\right)=\displaystyle F_{jt}-\frac{A_{jt}}{A_j}F_j-\d_{n+3}f_j\left[M^\prime u\right]_j(x,t)
	+\frac{A_{jj}}{A_j}v_j^2
	\\ [3mm]
	\displaystyle \qquad
\displaystyle+\bigg[F_{jj}+\frac{A_{jt}}{A_j}-\frac{F_jA_{jj}}{A_j}+\frac{1}{A_j}\sum_{k\not=j}A_{jk}v_k\bigg]v_j-\frac{F_j}{A_j}\sum_{k\not=j}A_{jk}v_k+\sum_{k\not=j}F_{jk}v_k-\d_{n+3}f_j\left[Mv\right]_j(x,t).	
\end{array}
\ee
Again, since $u$ is fixed, we suppress its dependence in the notation of 
$\widetilde f_j$.
Therefore, the system \reff{quasilin}, \reff{bc}, \reff{per} obtained
by differentiation with respect to $t$ takes the form \reff{quasilin_t}, \reff{bc_t}, \reff{per_t}. 
This system is semilinear
in the unknown $v$ and is of the form \reff{semilin}, \reff{bc}, \reff{per}. Moreover,  $v$ is a classical $C^1$-solution
to this  system.

We prove Theorem~\ref{thm:quasilin} in $l-2$ steps,  applying Theorem  \ref{thm:semilin} at each step. In order to verify the nonresonance conditions \reff{nonrez_semilin}, we first compute
\beq\label{not4}
\begin{array}{rcl}
	\displaystyle\d_{v_j}\widetilde f_j(x,t,v,q)&=&\displaystyle 2\frac{A_{jj}}{A_j}v_j+F_{jj}+\frac{A_{jt}}{A_j}-\frac{F_jA_{jj}}{A_j}+\frac{1}{A_j}\sum_{k\not=j}A_{jk}v_k.
\end{array}
\ee
Hence, the nonresonance conditions \reff{nonrez_semilin} for all
integers 
$i$ satisfying $2\le i\le l-1$, when applied to the system
\reff{quasilin_t}, \reff{bc_t}, \reff{per_t}, are precisely the
conditions \reff{nonrez_quasilin} for the corresponding integers $i$
satisfying $3\le i\le l$.

In the first step we show that
$
u\in C^3(\Pi;\mathbb R^n).
$
To this end, we prove that the $C^1$-solution $v$ to the problem
\reff{quasilin_t}, \reff{bc_t}, \reff{per_t}
actually belongs to $C^2$. Note that all coefficients in \reff{quasilin_t} are of class $C^2$. 
Now,  the regularity assumptions on the coefficients of 
 (\ref{quasilin}) and (\ref{bc}) imply that the coefficients of \reff{quasilin_t} and \reff{bc_t} satisfy the regularity assumptions of Theorem~\ref{thm:semilin} with~$l=2$. 
Taking into account the nonresonance condition
\reff{nonrez_quasilin}   for $i=3$ (equivalently, \reff{nonrez_semilin} for $i=2$), Theorem  \ref{thm:semilin} guarantees the $C^2$-regularity of $v=u_t$.
The desired $C^3$-regularity of $u$ now follows from the system 
(\ref{quasilin}), differentiated once with respect to $x$.
 
 Proceeding by induction,  we assume that $v\in C^{l-2}(\Pi;\R^n)$ and  $u\in C^{l-1}(\Pi;\R^n)$ and prove that 
 $u\in C^{l}(\Pi;\R^n)$. Accordingly to our line of the proof, we first show that
 $v\in C^{l-1}(\Pi;\R^n)$.
Under the inductive assumption (which implies the required regularity of $u$) and the nonresonance condition
 \reff{nonrez_quasilin}   for $i=l$ (equivalently, \reff{nonrez_semilin}   for $i=l-1$ ),  the coefficients of \reff{quasilin_t} and \reff{bc_t} satisfy the regularity assumptions of Theorem \ref{thm:semilin} required to obtain the $C^{l-1}(\Pi;\R^n)$-regularity of $v$.  Finally, differentiating \reff{quasilin_t} 
 $l-2$ times
with respect to $x$ yields
$
u\in C^{l}(\Pi;\R^n),
$
which completes the induction.

\section{The completely nonresonant setting}	\label{sec:nonres}
In this section, we characterize a subclass of the problems under consideration, the so-called completely nonresonant problems, for which neither the existence nor the higher regularity of time-periodic solutions requires any nonresonance conditions. Problems in this subclass are unconditionally nonresonant, in the sense that resonances, and hence small-divisor effects, are excluded independently of the particular coefficients or solutions involved.

We analyze this issue  in the linear setting, since, as shown in Subsections~\ref{2} and~\ref{3}, the semilinear and quasilinear cases reduce to the linear one.

As follows from the proof of Theorem~\ref{thm:lin} (in particular,
 Claims~\ref{I-C} and~\ref{I-C_1}) together with
  Lemma~\ref{lem:smoothing}, nonresonance conditions are naturally
   formulated in terms of the operator $\CC$ defined by \reff{CDF}. More
    generally, they are expressed through the operators $I-\CC^m$,  with
     their most general formulation given by Condition~1 of Lemma~\ref{lem:smoothing}. The explicit conditions \reff{nonrez} of Theorem~\ref{thm:lin} provide constructive sufficient criteria for Condition~1.

A distinguished class of completely nonresonant problems is characterized by the nilpotency of $\CC$, which implies Condition~1 of Lemma~\ref{lem:smoothing}.

By Theorem~1.7 of \cite{K-FTS}, nilpotency of $\CC$ is equivalent to finite-time stability of the corresponding subproblems. Together with the combinatorial and algebraic criteria for robust finite-time stability established in \cite[Theorems~1.6 and~1.11]{KL21}, this yields several sufficient conditions for complete nonresonance.
Consequently, combining Theorem~\ref{thm:lin} with these results, we obtain the following regularity theorem in the completely nonresonant setting.

\begin{thm}\label{ss22} Let $l\ge 1$ be arbitrary fixed integer.
	Assume that all conditions excepting \reff{nonrez} of Theorem \ref{thm:lin} are satisfied.
	Moreover, let $W=(w_{jk})$ be a  constant zero-one $n\times n$-matrix and   $r_{jk}(t)=\tilde r_{jk}(t)w_{jk}$,
	where $\tilde r_{jk}\in C_{per}^l$ for all $j,k\le n$. Assume that   the matrix $W$ fulfills one of the
	following  equivalent conditions:
	
	$(\io)$ the products $w_{i_1i_2}w_{i_2i_3}\cdots w_{i_{n}i_{n+1}}$  equal  zero for all tuples
	$\left(i_1,\dots,i_{n+1}\right)\in\left\{1,\dots,n\right\}^{n+1}$;
	
	$(\io\io)$ all principal minors of  the matrix $W$ equal zero;
	
$(\io\io\io)$ the matrix $W$ is nilpotent, with
	$W^n=0$;
	
	$(\io v)$ $W$ is the adjacency matrix of a directed acyclic graph.

Then	any continuous  solution to the problem (\ref{lin})--\reff{per} 
	belongs to~$C_{per}^l(\Pi;\R^n)$.
\end{thm}

\section{Further remarks and discussion}\label{nonrez-cond}

\subsection{Why nonresonance conditions appear}\label{why_nonrez}

In time-periodic linear hyperbolic systems with reflection boundary
conditions, solutions propagate along characteristic curves and are coupled at
the boundaries by reflection rules, possibly changing components.
Iteration of this
propagation-reflection mechanism produces a return dynamics on
$[0,1]\times(\mathbb R/2\pi\mathbb Z)$. As waves complete round-trips through this feedback loop, they accumulate phase shifts from both propagation and boundary interactions.
Resonances occur when temporal
oscillations synchronize with these  accumulated phases,  resulting in wave reinforcement and the formation of sharp resonant modes.

This mechanism is most transparent in the autonomous linear case. A
time-periodic solution can then be expanded into Fourier modes,
$$
u(x,t)=\sum_{k\in\mathbb Z} u^{(k)}(x)e^{ik\omega t}.
$$
After substitution into the system, one obtains a family of boundary value
problems for the Fourier coefficients $u^{(k)}$. The spectral parameter
$ik\omega$ enters through
$$
\partial_t u(x,t)
=
\sum_{k\in\mathbb Z} ik\omega\,u^{(k)}(x)e^{ik\omega t}.
$$
Thus, for each $k\neq0$, solvability and regularity depend on the
invertibility properties of the corresponding Fourier-mode operator. Both cases exhibit resonance, characterized either by a failure of operator invertibility for certain $k$ (exact resonance) or by the rapid growth of the inverse operator as $|k|\to\infty$ (approximate resonance).  The latter
produces small divisors in the formulas for $u^{(k)}$. In more detail, the coefficients
$u^{(k)}$ may then decay too slowly, even if the data are smooth. Since $m$
time derivatives require the convergence of
$$
\sum_{k\in\mathbb Z} (ik\omega)^m u^{(k)}(x)e^{ik\omega t},
$$
insufficient decay of $u^{(k)}$ obstructs $C^m$-regularity in time.

Nonresonance conditions exclude exact resonances and control approximate ones.
In the autonomous setting, they provide the bounds on the Fourier-mode problems
that yield the decay of $u^{(k)}$ required for the convergence of the
differentiated Fourier series.

The results of this paper show that the nonautonomous case is qualitatively
different. Because Fourier modes are not decoupled, regularity of the data alone
does not, in general, imply higher regularity of time-periodic solutions.
Additional nonresonance conditions are required, and their number is tied to
the desired order of regularity. The example in Subsection \ref{number} below shows that this
requirement is, in general, unavoidable.


\subsection{The case $n=2$}
If the system \reff{lin} (respectively, \reff{semilin} and \reff{quasilin}) consists of two equations, then the 
nonresonance conditions \reff{nonrez} (respectively, \reff{nonrez_semilin} and \reff{nonrez_quasilin}) can be refined into more explicit conditions.
In particular, if $n=2$, the coefficients $a_j$ and $b_{jj}$ are independent of $t$, and the boundary conditions are of the form
$$
u_1(0,t)=r_0u_2(0,t),\quad 	u_2(1,t)=r_1u_1(1,t),
$$
then \reff{nonrez} can be replaced by the weaker  condition (see \cite[Remark 1.6]{KR_separation})
$$
|r_0r_1|\ne\exp\left\{\int_{0}^{1}\left(\frac{b_{11}(x)}{a_1(x)}-\frac{b_{22}(x)}{a_2(x)}\right)\dd x\right\}.
$$
This is, in fact, the necessary and sufficient condition (see \cite[Remarks 1.6 and 1.9]{KR_separation}).
The analogue of the last condition in the nonautonomous setting, when  $a_j$, $b_{jj}$, $r_0$, and $r_1$ depend on~$t$,
takes the form (see \cite[Condition (3.5)]{KmKl})
$$
\begin{array}{cc}
	\displaystyle\exp{\int_{0}^{1}\left[\left (\frac{b_{22}}{a_2}\right )(\eta,\omega_2(\eta,0,\omega_1(0,1,t)))-\left (\frac{b_{11}}{a_1}\right )(\eta,\omega_1(\eta,1,t))\right] d\eta}
	\\ [6mm]
	\times\left|r_0(\omega_1(0,1,t))r_1(\omega_2(1,0,\omega_1(0,1,t)))\right|\neq 1 \quad\textrm{for all}\; t\in\R.
\end{array}
$$

\subsection{The role of the lower-order terms in the nonresonance conditions  \reff{nonrez}, \reff{nonrez_semilin}, and \reff{nonrez_quasilin}}\label{lower_terms}
The nonresonance conditions \reff{nonrez}, \reff{nonrez_semilin}, and~\reff{nonrez_quasilin} established in Theorems~\ref{thm:lin}, \ref{thm:semilin}, and~\ref{thm:quasilin} are sufficient conditions. They depend on the leading coefficients $a_j$, the diagonal zero-order coefficients $b_{jj}$ of the hyperbolic operator, and the reflection coefficients $r_{jk}$ appearing in the boundary operator.

The appearance of $a_j$ and $r_{jk}$ can be explained by the dissipative nature of the problems under consideration.
Specifically,  energy is dissipated along the characteristic curves defined by \reff{char} due to repeated reflections at the boundary.

What, then, is the role of the coefficients $b_{jj}$? The answer comes from the following simple calculation.
Consider the    linear system
\beq\label{kk2}
\partial_tu_j+a_j(x,t)\partial_xu_j+\sum_{k=1}^n b_{jk}(x,t)u_k= f_j(x,t),\quad j\le n,
\ee
\beq\label{kk3}
\begin{array}{ll}
	u_{j}(0,t)= [Ru]_j(t)+h_j(t), \quad 1\le j\le m,\\ [1mm]
	u_{j}(1,t)= [Ru]_j(t)+h_j(t), \quad m< j\le n.
\end{array}
\ee
We dropped here  integral contributions, since these do not influence resonance phenomena. 
Applying to \reff{kk2}--\reff{kk3} the (non-degenerate) change of variables $u_j\mapsto v_j$ given by
$$
u_j(x,t)=v_j(x,t)\exp\biggl\{\int^{1-x_j}_x\frac{b_{jj}(\eta,\om_j(\eta,x,t))}{a_j(\eta,\om_j(\eta,x,t))}\dd \eta\biggr\},\quad
j\le n,
$$
where $x_j$ are introduced in \reff{*k}, one can rewrite \reff{kk2}--\reff{kk3} in the following equivalent form, where the leading part of \reff{kk2} is preserved, the diagonal lower-order terms are eliminated, and all other  coefficients are transformed appropriately:
\beq\label{kk2v}
\partial_tv_j+a_j(x,t)\partial_xv_j+e_j^{-1}(x,t)\sum_{k\ne j}e_k(x,t)\, b_{jk}(x,t)v_k=e_j^{-1}(x,t)\,f_j(x,t),\quad j\le n,
\ee
\beq\label{kk3v}
\begin{array}{rcl}
	v_{j}(x_j,t)= e_j^{-1}(x_j,t)\, [Rv]_j(t)+e_j^{-1}(x_j,t)\,h_j(t), \quad j\le n.
\end{array}
\ee
Here
$$
e_j(x,t)=\exp\biggl\{\int^{1-x_j}_x\frac{b_{jj}(\eta,\om_j(\eta))}{a_j(\eta,\om_j(\eta))}\dd \eta\biggr\}, \quad j\le n.
$$
Observe that in the  system \reff{kk2v}--\reff{kk3v}, the diagonal 
coefficients $b_{jj}$ are  absorbed into the reflective boundary coefficients 
(new reflective coefficients in \reff{kk3v} take the form $e_j^{-1}(x_j,t)\,r_{jk}(t)$), which means that they also participate in the 
dissipative mechanism.

\subsection{Autonomy and the reduction to a single nonresonance condition} 
\label{autonom}
As stated in Theorems \ref{thm:lin}, \ref{thm:semilin}, and~\ref{thm:quasilin}, any higher order of regularity of solutions requires, in general, additional nonresonance conditions. More precisely, accordingly to 
\reff{nonrez} (and similarly for \reff{nonrez_semilin} and \reff{nonrez_quasilin}),
to obtain $C^1$-regularity of a continuous solution to the linear problem, one generally needs $n$ nonresonance conditions, while achieving $C^l$-regularity typically requires $nl$ such conditions.
There is, however, an important exception to this rule, namely the completely nonresonant case, in which no nonresonance conditions are required at all. This class of problems is described in Section~\ref{sec:nonres}.

Another exception concerns the explicit dependence of the leading coefficients $a_j$
on $t$. Specifically, if a coefficient 
$a_j$ is independent explicitly of $t$ for some  $j\in\{1,\dots,n\}$, then, accordingly to~\reff{nonrez}, the maximal number of nonresonance conditions needed for $C^l$-regularity
of a continuous solution in the linear setting  is reduced to $nl-(n-1)$, since the condition \reff{nonrez} corresponding to this index $j$ is the same for all $1\le i\le l$. 
In particular, if all coefficients $a_j$ are independent of $t$ (which includes the autonomous case),  not more than $n$ nonresonance conditions suffice to obtain arbitrary high regularity.
A similar observation is true in the semilinear and quasilinear settings.

In view of this observation, the following result follows as a consequence of Theorems \ref{thm:lin}, \ref{thm:semilin}, and 
\ref{thm:quasilin}.
\begin{cor}\label{x}\rm
	{\bf 1.} Let the condition \reff{aj} be satisfied, and suppose that
	the coeﬃcients $a_j$, $j\le n$, are independent of $t$.  Moreover, assume that the condition \reff{nonrez} 
	holds for $i=1$, and that all coefficients $a_j$, $b_{jk}$, $f_j$, $m_{jk}$, $h_j$, $r_{jk}$, and $p_{jk}$ are
	$C^\infty$-smooth in their arguments. Then any continuous solution to the linear problem 
	(\ref{lin})--\reff{per} is $C^\infty$-smooth.\\
	{\bf 2.}  
	Let the condition \reff{aj} be satisfied, and suppose that the coefficients $a_j$, $j\le n$,  are independent of $t$. 
	Let $u$ be a classical solution to 
	the semilinear problem \reff{semilin}, \reff{bc}, \reff{per}. Moreover, assume that the condition \reff{nonrez_semilin} 
	holds for $i=2$, and that
	all coefficients $a_j$, $b_{jk}$, $f_j$, $m_{jk}$, $h_j$, and $p_{jk}$ are
	$C^\infty$-smooth in their arguments. Then $u$ is $C^\infty$-smooth.\\
	{\bf 3.} 
	Let the condition \reff{aj_quasilin} be satisfied, and suppose that the coefficients $a_j$, $j\le n$,  do not depend explicitly on $t$. 
	Let	 $u$ be a classical $C^2$-solution to 
	the quasilinear problem \reff{semilin}, \reff{bc}, \reff{per}. Moreover, assume that the condition \reff{nonrez_quasilin} is satisfied
	for $i=3$, and that  all coefficients $a_j$, $b_{jk}$, $f_j$, $m_{jk}$, $h_j$, and $p_{jk}$ are
	$C^\infty$-smooth in their arguments. Then $u$ is $C^\infty$-smooth.
\end{cor}

\subsection{Essential role of additional nonresonance conditions in the nonautonomous setting}\label{number}
Here we show that the conditions \reff{nonrez} (and similarly for  \reff{nonrez_semilin} and \reff{nonrez_quasilin}), although sufficient,  are in general essentially necessary for higher regularity.

We start with a preliminary explanation. Differentiating an {\it autonomous} linear hyperbolic system $i$ times with respect to $t$ yields a system for $\d_t^iu$ with the same principal part, the same diagonal lower-order terms, and the same reflective boundary coefficients.
By contrast, in the {\it nonautonomous} case this procedure produces a shifted diagonal lower-order part of the hyperbolic operator. In particular, the diagonal coefficients $b_{jj}$ are replaced by   $b_{jj}-i\frac{\d_ta_j}{a_{j}}$ for $j\le n$, reflecting the explicit time dependence of the coefficients~$a_j$. 

We will, therefore, concentrate on the nonautonomous leading coefficients $a_j$,   when at least one of the coefficients $a_j$ depends on $t$.  To this end, we revisit  the example of the linear problem (\ref{lin})--\reff{per} discussed  in a similar context in \cite[Subsection 3.6]{KRT_evol}.  Specifically, we consider the system
\begin{eqnarray}
	& & \partial_t u_1  + \frac{2}{4\pi-1}\partial_xu_1  = 1, \quad \partial_t u_2  - (2+\sin t)\partial_x u_2  = 0,   
	\nonumber\\ 
	& & u_j(x,t+2\pi) = u_j(x,t),\quad j=1,2,\label{ex}\\ 
	& & u_1(0,t)=r_1(t) u_2(0,t),\quad u_2(1,t)=r_2 u_1(1,t), \nonumber
\end{eqnarray}
where the $2\pi$-periodic function $r_1$ and the  constant $r_2$ are such that
\beq\label{rr}
0<\sup_{t\in\R}r_1(t)<1, \quad 0 < r_2 < 1.
\ee

We first show that the problem \reff{ex}, under the condition \reff{rr}, admits a unique continuous solution. We then prove that this solution fails to be $C^1$-regular unless the nonresonance condition \reff{nonrez} for $i=1$ holds. Note that the condition \reff{rr} does not imply \reff{nonrez} for $i=1$. In other words, we will show that the condition \reff{nonrez} is not necessary for the existence of a continuous solution but it is essential for the existence of a $C^1$-solution.

The characteristics of \reff{ex} are given by the formulas
\begin{eqnarray*}
	& & \om_1(\xi,x,t)=\frac{4\pi-1}{2}(\xi-x)+t, \\
	& & \om_2(\xi,x,t)=p^{-1}(p(t)+\xi-x),
\end{eqnarray*}
where $p(t)=-2t+\cos t$. Moreover,
\beq\label{der}
\begin{array}{rcl}
	\displaystyle\d_t\om_2(\xi,x,t)&= &\displaystyle\exp \int_\xi^x \left(\frac{a^\prime}{a^2}\right)(\om_2(\eta,x,t)) \dd \eta  \\ [3mm]
	\displaystyle   &=& \displaystyle \exp \int_\xi^x\frac{\dd}{\dd\eta}\ln a(\om_2(\eta,x,t))\dd \eta\
	=\frac{a(t)}{a(\om_2(\xi,x,t))} ,
\end{array}
\ee
where $a(t)=-2-\sin t$.
Then, integrating  the differential system   $\reff{ex}_1$ along the characteristic curves and using the boundary conditions $\reff{ex}_3$ leads to the following system of functional equations:
\begin{eqnarray}
	& & u_1(x,t)=r_1\left( t - \frac{4\pi-1}{2}x\right)u_2\left(0, t - \frac{4\pi-1}{2}x\right)+
	\frac{4\pi-1}{2}x,\label{1a} \\
	& & u_2(x,t)=r_2 u_1(1,p^{-1}(p(t) + 1 - x)). \label{2a}
\end{eqnarray}
By the Banach fixed point argument, the conditions \reff{rr} are sufficient to ensure existence and uniqueness of a continuous time-periodic solution to \reff{ex}.

Inserting \reff{1a} into \reff{2a}, we get
\beq
\label{3a}
\begin{array}{rcl}\displaystyle
	u_2(0,t)&=&\displaystyle r_2 r_1\left(p^{-1}(p(t)+1) -\frac{4\pi-1}{2}\right)\\ &&\times\displaystyle u_2\left(0, p^{-1}(p(t)+1) -\frac{4\pi-1}{2}\right) 
	+ r_2 \frac{4\pi-1}{2}.
\end{array}
\ee
Using the $2\pi$-periodicity of $u_2$ in $t$, we now determine the  values of $t$ for which $u_2$  has the same argument
on both sides of \reff{3a}. This occurs, for instance, if
$t - 2\pi = p^{-1}(p(t)+1) - (4\pi-1)/2$.
This equality holds if and only if $p(t)+1=p(t - \frac{1}{2})$, or
equivalently,
$$
\cos t-\cos \left(t - \frac{1}{2}\right)=-2\sin\left(t - \frac{1}{4}\right)\sin\left(\frac{1}{4}\right)=0.
$$
The last equation has the solutions $1/4+\pi k$, $k\in\Z$. Set $t_0=1/4$.
Then, due to \reff{rr}, the equation
\reff{3a} yields
\beq
\label{t0}
u_2(0,t_0)=r_2\ \frac{4\pi-1}{2(1-r_2r_1(t_0))}\ne 0.
\ee
Notice the obvious identity
$p^{-1}(p(t)+1)=\om_2(1,0,t)$. Further, if the derivative $\d_tu_2(0,t_0)$ exists, then it is given by the formula
\beq
\label{8a}
\d_tu_2(0,t_0)=r_2 r_1(t_0)\d_t\om_2(1,0,t_0)\d_tu_2(0,t_0)+r_2r_1^\prime(t_0)\d_t\om_2(1,0,t_0)u_2(0,t_0).
\ee
Due to \reff{der}, we have
\begin{eqnarray*}
	\d_t\om_2(1,0,t_0)=\frac{a(t_0)}{a(\om_2(1,0,t_0))}=
	\frac{-2-\sin(1/4)}{-2-\sin(-1/4)}>1.
\end{eqnarray*}
We can choose a constant $r_2$ and  a smooth $2\pi$-periodic function $r_1(t)$ such that,
in addition to  the condition \reff{rr}, it holds that
\beq
\label{10a}
r_2 r_1(t_0)\d_t\om_2(1,0,t_0) = 1 \mbox{ and } r_1^\prime(t_0)\not=0,
\ee
contradicting  \reff{t0}--\reff{8a}.
Consequently, the continuous solution to \reff{3a}, and therefore to \reff{1a}--\reff{2a},
fails to be differentiable at $t=t_0$.

The failure of the condition \reff{nonrez} for $i=1$
can be checked directly. Indeed, in view of \reff{10a}, for any $\psi\in BC(\R,\R^2)$ satisfying $\|\psi\|_{BC}=1$ and
$\psi_1(\omega_2(1,0,t_0))=1$, one has
\begin{eqnarray*}
	\lefteqn{ \|G_1\|_{\LL(BC(\R,\R^2))}\ge |(G_1\psi)_2(t_0)|
		= c_2^1(1,0,t_0)|(\widetilde R\psi)_2(\omega_2(1,0,t_0))|} \\ &&
	= c_2^1(1,0,t_0)r_2|\psi_1(\omega_2(1,0,t_0))| = r_2\exp \biggl\{\int_0^1\left(- \frac{a^\prime(\om_2(\eta,0,t_0))}{a(\om_2(\eta,0,t_0))^2}\right)\dd\eta\biggr\}\\
	& &  =  r_2\exp \biggl\{\int_1^0\frac{\dd}{\dd\eta}\ln a(\om_2(\eta,0,t_0))\dd\eta\biggr\}
	=  r_2\frac{a(t_0)}{a(\om_2(1,0,t_0))}
	=  r_2\d_t\om_2(1,0,t_0)>1,
\end{eqnarray*}
as desired.

The essential role of further nonresonance conditions in achieving higher regularity can be shown analogously.

\subsection{More about the number of nonresonance conditions}\label{number_more}

As follows from the proof of Theorem~\ref{thm:lin} in Section~\ref{regularity},  in particular, from the proof of Claim~\ref{I-C_1} in Subsection~\ref{1}, if a starting solution of the problem (\ref{lin})--\reff{per}  is not merely continuous but already possesses, say, $C^{l_0}_{per}$-regularity, then in order to obtain
$C^{l}_{per}$-regularity  with $l>l_0$, it suffices to impose no more than $(l-l_0)n$ nonresonance conditions. More precisely, in this case we require that the nonresonance conditions \reff{nonrez} be satisfied for $l_0+1\le i\le l$. A similar remark applies to  Theorems \ref{thm:semilin} and \ref{thm:quasilin}. 

\subsection{Nonresonance conditions when $a_j$ is independent of $x$}
\label{nonautonom}

Let us examine the condition \reff{nonrez} in the linear setting, in the special
case where the characteristic speeds $a_j$ do not depend explicitly on $x$.
Using the following calculations:
\begin{eqnarray*}
	\lefteqn{ \exp\biggl\{-i\int_{1 - x_j}^{x_j}
		\biggl(\frac{\d_ta_j}{a_{j}^2}\biggr)(\om_j(\eta,1 - x_j,t))\dd\eta\biggr\}} \\ &&=
	\exp\biggl\{i	\int_{x_j}^{1-x_j}\frac{\dd}{\dd\eta}\ln a_j(\eta,1 - x_j,t))\dd\eta\biggr\}
	=\left(\frac{a_j(t)}{a_j(\om_j(x_j,1 - x_j,t))}\right)^i,
\end{eqnarray*} 
this condition can be written in the more
transparent form
\beq\label{nonrez-x}
\left(\frac{a_j(t)}
{a_j(\omega_j(x_j,1-x_j,t))}\right)^i
\exp\left\{
\int_{1-x_j}^{x_j}
\left(\frac{b_{jj}}{a_j}\right)
(\eta,\omega_j(\eta,1-x_j,t))\,d\eta
\right\}
\sum_{k=1}^n
\left|r_{jk}\left(\omega_j(x_j,1-x_j,t)\right)\right|
<1
\ee
for all $t\in\mathbb R$ and for the relevant values of $i$, depending on
the desired order of regularity.

By \reff{aj}, the quotient
$$
\tilde a_j(t):=\frac{a_j(t)}{a_j(\omega_j(x_j,1-x_j,t))}
$$
is positive for all $t\in\mathbb R$ and $j\le n$. If $a_j$  does not also depend on $t$ for some $j\le n$, this quotient
is identically equal to one for this $j¸$. If $a_j$ is genuinely time-dependent, however, $\tilde a_j$
is in general not bounded from above by one. In particular,
whenever $\tilde a_j(t_0)>1$
for some $t_0$, the factor $\tilde a_j(t_0)^i$
grows exponentially as $i\to\infty$.

Consequently, for any fixed  $l_0\in\mathbb N$, condition
\reff{nonrez-x} may still hold for all $1\leq i\leq l_0$, provided the
remaining factors are sufficiently small. This can occur, for instance, if the
reflection coefficients $r_{jk}$ are sufficiently small, or if the exponential
factor in \reff{nonrez-x} is sufficiently damping. However, for fixed
coefficients, the presence of a quotient larger than one prevents
\reff{nonrez-x} from holding for arbitrarily large $i$.

This shows that, in contrast with the autonomous case for the coefficients
$a_j$ (see Corollary~\ref{x}), the nonautonomous case is qualitatively
different: one cannot in general expect arbitrarily high regularity of
continuous solutions of the linear problem (\ref{lin})--\reff{per} from a single nonresonance condition. Similar considerations apply
to the semilinear and quasilinear settings.

\subsection{About weak hyperbolicity} \label{weak} Our higher regularity results can also be established by our approach for a class of weakly hyperbolic systems, in which weak hyperbolicity is compensated by Levy-type conditions, such as those in \cite[Condition~(1.5)]{KR_regul_ex}.

\bibliographystyle{abbrv}
\bibliography{my_bibfile}

@book{Bertero98,
	author    = {Bertero, M. and Boccacci, P.},
	title     = {Introduction to Inverse Problems in Imaging},
	publisher = {CRC Press},
	year      = {1998}
}

@book{BertiBolle2020,
	author    = {Massimiliano Berti and Philippe Bolle},
	title     = {Quasi-Periodic Solutions of Nonlinear Wave Equations on the {$d$}-Dimensional Torus},
	series     = {EMS Monographs in Mathematics},
	publisher  = {European Mathematical Society Publishing House},
	address    = {Berlin},
	year       = {2020},
	isbn       = {978-3-03719-211-5},
	doi        = {10.4171/211}
}

@article{inv-waves,
	author  = {Blazek, K. and Stolk, C. and Symes, W.},
	title   = {A mathematical framework for inverse wave problems in heterogeneous media},
	journal = {Inverse Problems},
	volume  = {29},
	number  = {6},
	pages   = {065001},
	year    = {2013}
}

@article{Coron_Fr,
	author  = {Coron, J.-M. and Hu, L. and Olive, G. and Shang, P.},
	title   = {Boundary stabilization in finite time of one-dimensional linear hyperbolic balance laws with coefficients depending on time and space},
	journal = {Journal of Differential Equations},
	volume  = {271},
	pages   = {1109--1170},
	year    = {2021}
}

@article{CraigWayne,
	author  = {Craig, W. and Wayne, C. E.},
	title   = {Newton's method and periodic solutions of nonlinear wave equations},
	journal = {Communications on Pure and Applied Mathematics},
	volume  = {46},
	pages   = {1409--1498},
	year    = {1993}
}

@incollection{Dashti17,
	author    = {Dashti, M. and Stuart, A. M.},
	title     = {The Bayesian approach to inverse problems},
	booktitle = {Handbook of Uncertainty Quantification},
	publisher = {Springer},
	year      = {2017}
}

@article{Friedrichs,
	author  = {Friedrichs, K. O.},
	title   = {Nonlinear hyperbolic differential equations for functions of two independent variables},
	journal = {American Journal of Mathematics},
	volume  = {70},
	pages   = {555--589},
	year    = {1948}
}

@article{Gerken2,
	author  = {Gerken, T.},
	title   = {Dynamic inverse wave problems -- {Part II}: Operator identification and applications},
	journal = {Inverse Problems},
	volume  = {36},
	number  = {2},
	pages   = {024005},
	year    = {2020}
}

@article{Gerken1,
	author  = {Gerken, T. and Gr{\"u}tzner, S.},
	title   = {Dynamic inverse wave problems -- {Part I}: Regularity for the direct problem},
	journal = {Inverse Problems},
	volume  = {36},
	number  = {2},
	pages   = {024004},
	year    = {2020}
}

@book{Hale,
	author    = {Hale, J. K. and Magalhaes, L. T. and Oliva, W. M.},
	title     = {An Introduction to Infinite Dimensional Dynamical Systems: Geometric Theory},
	publisher = {Springer-Verlag},
	year      = {1984}
}

@article{Hale_Scheurle,
	author  = {Hale, J. K. and Scheurle, J.},
	title   = {Smoothness of bounded solutions of nonlinear evolution equations},
	journal = {Journal of Differential Equations},
	volume  = {56},
	pages   = {142--163},
	year    = {1985}
}

@article{Hillen,
	author  = {Hillen, T. and Rohde, C. and Lutscher, F.},
	title   = {Existence of weak solutions for a hyperbolic model of chemosensitive movement},
	journal = {Journal of Mathematical Analysis and Applications},
	volume  = {260},
	pages   = {173--199},
	year    = {2001}
}

@book{Horm,
	author    = {H{\"o}rmander, L.},
	title     = {The Analysis of Linear Partial Differential Operators {I}: Distribution Theory and Fourier Analysis},
	series    = {Grundlehren der mathematischen Wissenschaften},
	volume    = {256},
	publisher = {Springer-Verlag},
	year      = {1990}
}

@book{Ianelli,
	author    = {Iannelli, M.},
	title     = {Mathematical Theory of Age-Structured Population Dynamics},
	series    = {Applied Mathematics Monographs},
	publisher = {CNR, Giardini Editori e Stampatori},
	address   = {Pisa},
	year      = {1995}
}

@book{Inaba,
	author    = {Inaba, H.},
	title     = {Age-Structured Population Dynamics in Demography and Epidemiology},
	publisher = {Springer},
	year      = {2017}
}

@book{Kielhofer,
	author    = {Kielh{\"o}fer, H.},
	title     = {Bifurcation Theory: An Introduction with Applications to PDEs},
	series    = {Applied Mathematical Sciences},
	volume    = {156},
	publisher = {Springer},
	year      = {2004}
}

@book{Kirsch11,
	author    = {Kirsch, A.},
	title     = {An Introduction to the Mathematical Theory of Inverse Problems},
	edition   = {2},
	publisher = {Springer},
	year      = {2011}
}

@incollection{K-FTS,
	author    = {Kmit, I.},
	title     = {1D Hyperbolic Systems with Nonlinear Boundary Conditions {II}: Criteria for Finite Time Stability},
	booktitle = {Analysis, Applications, and Computations},
	series    = {Trends in Mathematics},
	publisher = {Birkh{\"a}user},
	address   = {Cham},
	pages     = {439--453},
	year      = {2023}
}

@article{KL21,
	author  = {Kmit, I. and Lyul'ko, N.},
	title   = {Finite time stabilization of nonautonomous first order hyperbolic systems},
	journal = {SIAM Journal on Control and Optimization},
	volume  = {59},
	number  = {5},
	pages   = {3179--3202},
	year    = {2021}
}

@article{KmKl,
	author  = {Kmit, I. and Klyuchnyk, R.},
	title   = {Fredholm solvability of time-periodic boundary value hyperbolic problems},
	journal = {Journal of Mathematical Analysis and Applications},
	volume  = {442},
	number  = {2},
	pages   = {804--819},
	year    = {2016}
}

@article{KR_regul_ex,
	author  = {Kmit, I. and Recke, L.},
	title   = {Fredholm alternative and solution regularity for time-periodic hyperbolic systems},
	journal = {Differential and Integral Equations},
	volume  = {29},
	number  = {11/12},
	pages   = {1049--1070},
	year    = {2016}
}

@article{KR_separation,
	author  = {Kmit, I. and Recke, L.},
	title   = {Fredholmness and smooth dependence for linear time-periodic hyperbolic systems},
	journal = {Journal of Differential Equations},
	volume  = {252},
	number  = {2},
	pages   = {1962--1986},
	year    = {2012}
}

@article{KR_Hopf,
	author  = {Kmit, I. and Recke, L.},
	title   = {Hopf bifurcation for semilinear dissipative hyperbolic systems},
	journal = {Journal of Differential Equations},
	volume  = {257},
	number  = {1},
	pages   = {264--309},
	year    = {2014}
}

@article{KR_autonom,
	author  = {Kmit, I. and Recke, L.},
	title   = {Regularity of time-periodic solutions to autonomous semilinear hyperbolic PDEs},
	journal = {Journal of Mathematical Analysis and Applications},
	volume  = {529},
	number  = {1},
	pages   = {127562},
	year    = {2024}
}

@article{KR_IFT,
	author  = {Kmit, I. and Recke, L.},
	title   = {Solution regularity and smooth dependence for abstract equations and applications to hyperbolic PDEs},
	journal = {Journal of Differential Equations},
	volume  = {259},
	pages   = {6287--6337},
	year    = {2015}
}

@article{KRT_evol,
	author  = {Kmit, I. and Recke, L. and Tkachenko, V.},
	title   = {Bounded and almost periodic solvability of nonautonomous quasilinear hyperbolic systems},
	journal = {Journal of Evolution Equations},
	volume  = {21},
	pages   = {4171--4212},
	year    = {2021}
}

@article{LRR,
	author  = {Lichtner, M. and Radziunas, M. and Recke, L.},
	title   = {Well-posedness, smooth dependence and center manifold reduction for a semilinear hyperbolic system from laser dynamics},
	journal = {Mathematical Methods in the Applied Sciences},
	volume  = {30},
	pages   = {931--960},
	year    = {2007}
}

@article{Pavel,
	author  = {Pavel, L.},
	title   = {Classical solutions in Sobolev spaces for a class of hyperbolic Lotka--Volterra systems},
	journal = {SIAM Journal on Control and Optimization},
	volume  = {51},
	number  = {3},
	pages   = {2132--2151},
	year    = {2013}
}

@incollection{RW,
	author    = {Radziunas, M. and W{\"u}nsche, H.-J.},
	title     = {Dynamics of multisection DFB semiconductor lasers: traveling wave and mode approximation models},
	editor    = {Piprek, J.},
	booktitle = {Optoelectronic Devices: Advanced Simulation and Analysis},
	publisher = {Springer},
	address   = {Berlin},
	pages     = {121--150},
	year      = {2005}
}

@article{RM74,
	author  = {Rauch, J. B. and Massey, F. J. I.},
	title   = {Differentiability of solutions to hyperbolic initial-boundary value problems},
	journal = {Transactions of the American Mathematical Society},
	volume  = {189},
	pages   = {303--318},
	year    = {1974}
}

@book{Webb,
	author    = {Webb, G. F.},
	title     = {Nonlinear Age-Dependent Population Dynamics},
	publisher = {Marcel Dekker},
	year      = {1985}
}

@book{Wloka,
	author    = {Wloka, J.},
	title     = {Partial Differential Equations},
	publisher = {Cambridge University Press},
	year      = {1987}
}

@book{AP,
  author		= "Amerio, L. and Prouse, G.",
  title			= "Almost-periodic functions and functional equations",
  address		= "New York",
  publisher		= "Springer Science and Business Media",
  year			= "1971",
  doi			= "10.1007/978-1-4757-1254-4"
}

@article{BertiBiascoProcesi14,
	author  = {Berti, Massimiliano and Biasco, Luca and Procesi, Michela},
	title   = {KAM for Reversible Derivative Wave Equations},
	journal = {Archive for Rational Mechanics and Analysis},
	volume  = {212},
	number  = {3},
	pages   = {905--955},
	year    = {2014},
	doi     = {10.1007/s00205-014-0726-0}
}

@article{CraigWayne93,
	author  = {Craig, Walter and Wayne, C. Eugene},
	title   = {Newton's method and periodic solutions of nonlinear wave equations},
	journal = {Communications on Pure and Applied Mathematics},
	volume  = {46},
	number  = {11},
	pages   = {1409--1498},
	year    = {1993},
	doi     = {10.1002/cpa.3160461102}
}

@misc{MelnykPopRohde26,
	author        = {Mel'nyk, Taras and Pop, Sorin and Rohde, Christian},
	title         = {Multiscale Hyperbolic-Parabolic Models for Nonlinear Reactive Transport in Heterogeneously Fractured Porous Media},
	year          = {2026},
	eprint        = {2602.16439},
	archivePrefix = {arXiv},
	primaryClass  = {math.AP}
}

@article{Poschel96,
	author  = {P{\"o}schel, J{\"u}rgen},
	title   = {A KAM-theorem for some nonlinear partial differential equations},
	journal = {Annali della Scuola Normale Superiore di Pisa. Classe di Scienze},
	series  = {4},
	volume  = {23},
	number  = {1},
	pages   = {119--148},
	year    = {1996}
}

@article{Wayne90,
	author  = {Wayne, C. Eugene},
	title   = {Periodic and quasi-periodic solutions of nonlinear wave equations via KAM theory},
	journal = {Communications in Mathematical Physics},
	volume  = {127},
	number  = {3},
	pages   = {479--528},
	year    = {1990}
}

\end{document}